\documentclass[a4paper,11pt,reqno]{amsart}
\usepackage{amsmath}
\usepackage{relsize}
\usepackage[noadjust]{cite}
\usepackage{color}
\usepackage[dvipsnames]{xcolor}
\usepackage{mathtools}
\usepackage[normalem]{ulem}
\newcommand\redout{\bgroup\markoverwith
{\textcolor{red}{\rule[.5ex]{2pt}{0.4pt}}}\ULon}

\usepackage{hyperref, enumitem}
\hypersetup{
  colorlinks   = true, %Colours links instead of ugly boxes
  urlcolor     = blue, %Colour for external hyperlinks
  linkcolor    = Purple, %Colour of internal links%
  citecolor   = red %Colour of citations
}
\usepackage{amsmath,amsthm,amssymb}

\usepackage[margin=2.5cm]{geometry}
\usepackage{caption} 
\usepackage{tikz-cd}
\usetikzlibrary{calc,fit,matrix,arrows,automata,positioning}
\usetikzlibrary{decorations.pathreplacing,angles,quotes}
\usepackage{stmaryrd}

\usetikzlibrary{fit}
\tikzset{%
  highlight/.style={rectangle,rounded corners,fill=red!15,draw,
    fill opacity=0.5,thick,inner sep=0pt}
}
\newcommand{\tikzmark}[2]{\tikz[overlay,remember picture,
  baseline=(#1.base)] \node (#1) {#2};}
\newcommand{\Highlight}[1][submatrix]{%
    \tikz[overlay,remember picture]{
    \node[highlight,fit=(left.north west) (right.south east)] (#1) {};}
}

\usepackage{array}
\newcolumntype{x}[1]{>{\centering\arraybackslash\hspace{0pt}}p{#1}}

\theoremstyle{definition}
\newtheorem{theorem}{Theorem}[section]
\newtheorem{definition}[theorem]{{{Definition}}}
\newtheorem{example}[theorem]{{{Example}}}
\newtheorem{notation}[theorem]{{{Notation}}}
\newtheorem{remark}[theorem]{{{Remark}}}

\newtheorem{corollary}[theorem]{{{Corollary}}}%[theorem]
\newtheorem{proposition}[theorem]{{{Proposition}}}
\newtheorem{lemma}[theorem]{{{Lemma}}}

\newtheorem{construction}{{{Construction}}}

\newcommand{\C}{\mathcal C}
\newcommand{\Fl}{\mathfrak{F}l}
\newcommand{\D}{\mathcal{D}}

\newcommand{\F}{\mathbb F}

\DeclareMathOperator{\GL}{GL}
\DeclareMathOperator{\supp}{supp}

\DeclareMathOperator{\Mat}{Mat}

\DeclareMathOperator{\PG}{PG}
\DeclareMathOperator{\rk}{rk}

\DeclareMathOperator{\U}{U}
\DeclareMathOperator{\dd}{d}
\DeclareMathOperator{\ff}{fr}
\DeclareMathOperator{\rr}{r}
\DeclareMathOperator{\w}{w}
\DeclareMathOperator{\Gr}{Gr}
\DeclareMathOperator{\rs}{rowsp}
\DeclareMathOperator{\diag}{diag}
\DeclareMathOperator{\wt}{wt}
\DeclareMathOperator{\HH}{H}
\DeclareMathOperator{\Ss}{SS}

\newcommand{\df}{\dd_{\ff}}
\newcommand{\dr}{\dd_{\rr}}
\newcommand{\wf}{\w_{\ff}}
\newcommand{\MM}[2]{{#1}_{[#2]}}
\newcommand{\ntdf}{\{n,t,\delta\}_\F}
\newcommand{\st}{\,:\,}
\newcommand{\ntt}[1]{\{#1\}_{\F}}

\usepackage{scrextend}
\deffootnote{0em}{0em}{\thefootnotemark\quad}

\title{Maximum Flag-Rank Distance Codes}
\usepackage[foot]{amsaddr}
\author[G. N. Alfarano]{Gianira N. Alfarano}
\address{Gianira N. Alfarano, \textnormal{Department of Mathematics and Computer Science, Eindhoven University  of Technology, 
 De Groene Loper 5, 5612 AZ Eindhoven, the Netherlands}}
 \email{g.n.alfarano@tue.nl}
\author[A. Neri]{Alessandro Neri}
\address{Alessandro Neri, \textnormal{Department of Mathematics: Analysis, Logic and Discrete Mathematics, Ghent University, Krijgslaan 281, 9000
Gent, Belgium,  and Max-Planck-Institute for Mathematics in the Sciences, Inselstraße 22, 04103 Leipzig, Germany}}
 \email{alessandro.neri@ugent.be}
\author[F. Zullo]{Ferdinando Zullo}
\address{Ferdinando Zullo, \textnormal{Department of Mathematics and Physics, University of Campania ``Luigi Vanvitelli'', Viale Lincoln 5, 81100 Caserta, Italy}}
 \email{ferdinando.zullo@unicampania.it}

\begin{document}

\maketitle

\begin{abstract}
  In this paper we extend the study of linear spaces of upper triangular matrices endowed with the flag-rank metric. Such metric spaces are isometric to certain spaces of degenerate flags and have been suggested as suitable framework for network coding. In this setting we provide a Singleton-like bound which relates the parameters of a flag-rank-metric code. This allows us to introduce the family of maximum flag-rank distance codes, that are flag-rank-metric codes meeting the Singleton-like bound with equality. Finally, we provide several constructions of maximum flag-rank distance codes.
\end{abstract}

\medskip 
\noindent {\bf Keywords.} Network coding, degenerate flag variety, flag-rank-metric codes, upper triangular matrices.\\
{\bf MSC classification.} 94B65, 94B05, 15B99.

%%%%%%%%%%%%%%%%%%%%%%%%%%%%%%%%%%%%%%%%%%%%%%%%%%%%%%%%%%%%%%%%%%%%%%%%%%%%%%%%%%%%%%%%%%%%%%%%%%%%%%%%%%%%%%%%%%%%%%%%%%%%%%%%%%%%%%%%%%%%%%%%%%%%%%%%%%%%%%%%%%%%%%%%%%%%%%%%%%%%%%%%%%%%%%%%%%%%%

\section*{Introduction}
The traditional approach to network communication has undertaken a crucial change when the seminal paper \cite{ahlswede2000network} by Ahlswede \emph{et al.} came out with the idea that intermediate nodes can combine the inputs received  and forward the new messages. This networking technique is known nowadays as \emph{network coding}: transmitted data are encoded and decoded in order to increase network throughput, reduce delays and make the network robust. It regards situations where one source attempts to transmit several messages simultaneously to multiple terminals. 
In \cite{koetter2008coding}, K\"otter and Kschischang introduced the concept of transmitting information over a network encoded in subspaces. In this framework, the message alphabet is the set of all subspaces of a vector space $\F^n$ over a field $\F$, and a \emph{subspace code} is a set of subspaces of $\F^n$. The source sends a (basis of a) vector space, the receiver gathers a (basis of a) vector space possibly affected by noise. Subspace codes gained a lot of attention in the past years; see for instance \cite{horlemann2018constructions} and references therein. In many of the previous works, \emph{constant dimension} codes are considered, i.e. subspace codes in which all the codewords are vector subspaces having the same dimension. More formally, we have the following setup.
Let $n,k$ be positive integers with $1\leq k\leq n$, and denote by $\Gr_k(\F^n)$ the \emph{Grassmannian}, i.e. the set of all $k$-dimensional subspaces of $\F^n$. The Grassmannian $\Gr_k(\F^n)$ can be interpreted as a metric space endowed with the \emph{subspace distance}, defined for every $U,V\in\Gr_k(\F^n)$ as
$$ \dd_{\mathrm{S}}(U,V) = \dim_{\F}(U+V) -\dim_{\F}(U\cap V).$$
Notice that $ \Gr_k(\F^n)\cong \{A\in\Mat(k,n,\F) \st \rk(A)=k\}/\GL(k,\F)$, where $\Mat(k,n,\F)$ denotes the space of $k\times n$ matrices with entries in $\F$. In other words,  $\Gr_k(\F^n)$ can be represented by the set of $k\times n$ matrices of rank $k$ in \emph{reduced row echelon form}. Using this representation, we can naturally  partition $\Gr_k(\F^n)$  into \emph{cells}, according to the pivot positions of the matrices in reduced row echelon form. The largest cell in this decomposition is the set 
$$\Gr_k^0(\F^n)=\{U\in\Gr_k(\F^n) \st U=U_A=\rs(\; \mathrm{Id}_{k} \,\mid\, A \;), \; \textnormal{ for  }A\in\Mat(k,n-k,\F)\}.$$
In this cell we have that $\dd_{\mathrm{S}}(U_A,U_B) = 2\cdot\rk(A-B)$, that is, the metric spaces $(\Gr_k^0(\F^n),\dd_{\mathrm{S}})$  and $(\Mat(k,n-k,\F),2\cdot\dd_{\rk})$ are isometric, where $\dd_{\rk}$ is the distance function induced by the rank.
This point of view has the advantage  to add algebraic structure to constant dimension codes, which then can inherit an $\F$-vector space structure. This observation motivated a deep study of (linear) rank-metric codes, that is, linear spaces of matrices endowed with the rank distance; see \cite{sheekey201913} and references therein.
In \cite{liebhold2018network}, linear flag coding was suggested as substitute of subspace coding in linear network coding, where \emph{flag codes} can be seen as a generalization of constant dimension codes. 
This idea raised the interest of the scientific community, yielding to different research lines; see e.g. \cite{alonso2020flag, alonso2021optimum, alonso2021orbital, alonso2021cyclic, navarro2022flag, alonso2023flag,kurz2021bounds,stanojkovski2022submodule}.
More precisely, consider the \emph{flag variety} $\Fl(\F^{n+1})$ defined as the space 
$$\Fl(\F^{n+1})=\left\{(U_1,\ldots,U_{n})\in\prod_{i=1}^{n}\Gr_i(\F^{n+1}) \st U_1\subseteq U_2\subseteq\cdots\cdots \subseteq U_{n}\right\}.$$
A \emph{flag code} is a subset of $\Fl(\F^{n+1})$, whose codewords are hence flags  $(U_1,\ldots,U_{n})\in\Fl(\F^{n+1})$, where 
$U_i=\langle u_1,\ldots,u_i\rangle$ for all $i\in\{1,\ldots,n\}$.
Also $\Fl(\F^{n+1})$ can be partitioned into cells,  and the largest one is $\Fl^0(\F^{n+1})=\Fl(\F^{n+1})\cap \prod_{i=1}^{n}\Gr_i^0(\F^{n+1})$.
In \cite{fourier2021degenerate}, a slightly different variant has been proposed, where the flag variety is replaced by a natural degeneration, $\Fl^{(a)}(\F^{n+1})$, that we will define in more detail in Section \ref{sec:FRM}. This concept was introduced in \cite{feigin2012degeneration} under a pure algebraic perspective. From a linear algebra point of view, this degeneration has been studied in \cite{irelli2015degenerate}. 
In these works the authors illustrate also that  $\Fl^{(a)}(\F^{n+1})$ can be partitioned into cells, and the biggest cell in the decomposition is $\Fl^{(0a)}(\F^{n+1})=\Fl^{(a)}(\F^{n+1}) \cap \prod_{i=1}^{n}\Gr_i^0(\F^{n+1})$. 
 Furthermore, it is a metric affine space isometric to the space of upper triangular matrices endowed with the flag-rank metric; see Definition~\ref{def:frd} and Theorem~\ref{thm:isometry}. This observation motivates the development of a theory of flag-rank-metric codes in analogy to rank-metric codes, initiated by Fourier and Nebe in \cite{fourier2021degenerate}.

\medskip
\paragraph{\textbf{Our Contribution}}
The goal of this paper is to move forward on the theory of flag-rank-metric codes. These codes are linear subspaces of the space $\U(n,\F)$ of $n\times n$ upper triangular matrices over the field $\F$ endowed with the flag-rank distance. For a given $n \times n$ matrix $M$ and for $i \in \{1,\ldots,n\}$, denote by $M_{[i]}$ the  top-rightmost $i \times (n-i+1)$ submatrix of $M$. The flag-rank distance is induced by the flag-rank weight, which is defined on $\U(n,\F)$ as
$$ \wf(M)=\sum_{i=1}^n\rk(M_{[i]}). $$
In our treatment, as a first step we propose a Singleton-like bound that relates the parameters of a flag-rank-metric code; see Theorem \ref{th:Singleton_codimension}. In particular, we show that
for a linear subspace $\C\subseteq \U(n,\F)$ whose nonzero elements have all flag-rank weight at least $\delta$, it holds
$$\mathrm{codim}_{\F}(\C)\geq \delta-1+\lfloor{\sqrt{\delta-1}}\rfloor \cdot \left\lfloor{\frac{\sqrt{4(\delta-1)+1}-1}{2}}\right\rfloor.$$
This result extends the one obtained in \cite{fourier2021degenerate}, where the authors derived the bound only for the case $\delta=\left\lceil\frac{n}{2} \right\rceil\cdot \left\lceil\frac{n+1}{2} \right\rceil$.

 The existence of a Singleton-like bound raises the natural questions of existence and features of codes whose parameters meet the bound with equality, as it is common practice in algebraic coding theory. Therefore, we initiate  their study under the name of \emph{maximum flag-rank distance} (\emph{MFRD}) \emph{codes}.
  The only construction of MFRD codes known so far was given by Fourier and Nebe for $\delta=\left\lceil\frac{n}{2} \right\rceil\cdot \left\lceil\frac{n+1}{2} \right\rceil$ and $n$ odd. Using algebraic, geometric and combinatorial tools we obtain various new constructions of MFRD codes for more parameters. More precisely:
 \begin{enumerate}
     \item We extend Fourier and Nebe's construction of MFRD codes for $\delta=\left\lceil\frac{n}{2} \right\rceil\cdot \left\lceil\frac{n+1}{2} \right\rceil$  to $n$ even; see Construction \ref{constr:A}.
     \item We provide a construction of MFRD codes for $\delta=\left\lceil\frac{n}{2} \right\rceil\cdot \left\lceil\frac{n+1}{2} \right\rceil -1$ when $n$ is odd; see Construction \ref{constr:B}.
     \item We classify MFRD codes with $\delta=2$; see Section \ref{sec:d=2}.
     \item We characterize MFRD codes with $\delta=3$, introducing the geometric notion of \emph{support-avoiding} codes in the Hamming metric; see Construction \ref{constr:d=3}.
     \item We present a construction of MFRD codes with the aid of auxiliary maximum distance separable (MDS) codes in the Hamming metric and maximum rank distance (MRD) codes in the rank-metric; see Construction \ref{constr:C}.
 \end{enumerate}
These constructions are all based on different approaches and exploit a wide variety of tools: We use the theory of rank-metric codes and MRD codes, the theory of linear block codes in the Hamming metric and MDS codes, as well as several geometric objects such as caps in a  projective space. 

\medskip
\paragraph{\textbf{Outline}} The paper is organized as follows. In Section \ref{sec:preliminaries} we provide the necessary background material. In Section \ref{sec:FRM} we introduce flag-rank-metric codes. Section \ref{sec:singletonlikebound} is devoted to prove the Singleton-like bound for flag-rank-metric codes. In Section \ref{sec:MFRC} we exhibit several systematic constructions of MFRD codes. We conclude with some computational results and  open problems in Section \ref{sec:conclusion}.

\medskip

\paragraph{\textbf{Notation}} Throughout this paper $\F$ denotes any field, $n\geq 2$ is a fixed positive integer and $k:=\lceil\frac{n}{2}\rceil$. For $i \in \mathbb{N} =\{0,1,2, \ldots\}$ we let $[i]:=\{j \in \mathbb{N} \, : \, 1 \le j \le i\}$. We denote by $\Mat(\ell,m,\F)$ the space of $\ell \times m$ matrices with entries in $\F$ and, in case of square $\ell\times \ell$ matrices, we simply write $\Mat(\ell,\F)$. Moreover, $\U(n,\F)$ denotes the space of upper triangular $n\times n$ matrices with entries in  $\F$. 
For a matrix $M\in \Mat(n,m,\F)$ we denote by $\mathrm{rowsp}(M)$ the rowspace of $M$ over $\F$, that is the $\F$-subspace of $\F^m$ generated by the rows of $M$. 
If $A \in \Mat(n,m,\F)$, then $A_{I}^J$ denotes the submatrix formed by selecting the rows of $A$ indexed by $I \subseteq [n]$ and the columns of $A$ indexed by $J \subseteq [m]$. 
We denote by $\PG(t,\F)$ the $t$-dimensional projective space with underlined vector space $\F^{t+1}$.

\medskip

\paragraph{\textbf{Acknowledgement}}
G.~N.~A.\
is supported by the Swiss National Foundation through grant no. 210966. A.~N. is supported by   the FWO (Research Foundation Flanders) grant no. 12ZZB23N. 
F.Z. is supported by the project ``VALERE: VAnviteLli pEr la RicErca" of the University of Campania ``Luigi Vanvitelli'' and by the Italian National Group for Algebraic and Geometric Structures and their Applications (GNSAGA - INdAM). Moreover, he is very grateful for the hospitality of University of Zurich, Switzerland, and Max-Planck-Institute for Mathematics in the Sciences, Germany, where he was a visiting researcher for two weeks during the development of this research.

\section{Preliminaries}\label{sec:preliminaries}
In this section we provide the background for the rest of the paper. We first give a brief introduction to classic error-correcting codes and rank-metric codes. For a more detailed treatment on Hamming-metric codes and rank-metric codes we refer the interested reader to \cite{huffman2010fundamentals, sheekey201913, gorla2018codes}.

\subsection{Hamming-metric codes}
For any $v=(v_1,\ldots,v_n)\in\F^n$, we define the \textbf{Hamming support} of $v$ to be the set $\supp(v):=\{ i \in [n] \st v_i \neq 0\}$ and the \textbf{Hamming weight} of $v$ to be the quantity $\wt_{\HH}(v)=|\supp(v)|$. The latter naturally induces a metric on $\F^n$, called the \textbf{Hamming distance} and given by 
$$\dd_{\HH}(u,v):=\wt_{\HH}(u-v), \quad \mbox{ for every } u,v \in \F^n.$$  An $[n,k]_{\F}$ (\textbf{linear Hamming-metric}) \textbf{code} is an $\F$-linear subspace $\C \subseteq \F^n$ of dimension $k$ endowed with the Hamming metric. The vectors in $\C$ are called \textbf{codewords}. The \textbf{minimum Hamming distance} of $\C$ is defined as
$$\dd_{\HH}(\C):=\min\{\wt_{\HH}(c) \st c \in \C, \, c \neq 0\}.$$  If $d=\dd_{\HH}(\C)$ is known, we say that $\C$ is an $[n,k,d]_\F$ code. A matrix $G\in\Mat(k,n,\F)$ whose rows form a basis for $\C$ is called a \textbf{generator matrix} for $\C$.

A linear code $\C \subseteq \F^n$ cannot have simultaneously large minimum distance and large dimension. Several results in the literature capture this trade-off. One of the best-known relations among the parameters of an $[n,k,d]_\F$ code is undoubtedly the Singleton bound; see \cite{singleton1964maximum}. Let $\C \subseteq \F^n$ be an $[n,k,d]_\F$ code. The Singleton bound states that 
$$k \le n-d+1.$$
Codes whose parameters meet the Singleton bound with equality are called \textbf{maximum distance separable} or \textbf{MDS} for short and have the highest error-correction capability among the codes of same length and dimension. For an $[n,k,d]_\F$ code $\C$, we can define a parameter $s(\C) =n-d+1-k$ called \textbf{Singleton defect}, which measures how far  $\C$ is from being MDS.
Clearly, an MDS code has Singleton defect equal to $0$. A code $\C$ with $s(\C)=1$ is called \textbf{almost MDS}. For more details about these codes we refer the interested reader to \cite{de1996almost}.

Given an $[n,k]_\F$ code $\C$, we define its \textbf{dual code} $\C^\perp$ to be the dual space with respect to standard dot product, i.e.
$$ \C^\perp =\{v\in\F^n \st v\cdot c =0 \; \mbox{ for all } c\in \C\}.$$
A generator matrix $H\in \Mat(n-k,n,\F)$ for $\C^\perp$ is called \textbf{parity-check matrix} for the code~$\C$.

\subsection{Rank-metric codes}
Rank-metric codes have been introduced originally by Delsarte in~\cite{de78} and they have been intensively investigated in recent years because of their applications in crisscross error correction \cite{roth1991maximum}, cryptography \cite{gabidulin1991ideals} and network coding \cite{silva2008rank}; see \cite{bartz2022rank} for a survey on their applications. 
We endow the space $\Mat(n, m,\F)$ with the \textbf{rank metric}, defined as
\[\dr(A,B) = \mathrm{rk}\,(A-B), \qquad \mbox{ for every } A,B\in\Mat(n, m,\F).\]
 An $[n\times m,k]_\F$ \textbf{(linear) rank-metric code} $\C$ is a $k$-dimensional $\F$-subspace  of $\Mat(n,m,\F)$. The \textbf{minimum rank distance} of $\C$ is defined as
 \begin{align*}
     \dr(\C):= &\min\{ \dr(A,B) \st A,B \in \C,\,\, A\neq B \}\\
     = &\min\{ \rk(A) \st A \in \C,\,\, A\neq 0 \}.
 \end{align*}
If $d=\dr(\C)$ is known, we write that $\C$ is an $[n\times m,k,d]_\F$ code.
The parameters $n,m,k,d$ of an $[n\times m,k,d]_\F$ rank-metric code satisfy a Singleton-like bound \cite{de78}, that reads~as
\[ k \leq \max\{m,n\}(\min\{m,n\}-d+1). \]
When equality holds, we say that $\C$ is a \textbf{maximum rank distance} code or \textbf{MRD} for short.
Contrary to MDS codes which exist only over sufficiently large fields, it has been shown that MRD codes exist for any (admissible) choice of $n,m,d$,  over any finite field \cite{de78,ga85a}, and over any field admitting a degree $n$ or $m$ cyclic Galois extension \cite{guralnick1994invertible,roth1996tensor}.

For a rank-metric code $\C\subseteq\Mat(n,m,\F)$ and matrices $A\in \Mat(n,\F)$, and $B\in\Mat(m,\F)$, we define
\[ A\C=\{AC \st C \in \C\}, \quad \C B=\{CA \st C\in \C\}.\]
We recall the following two codes associated with $\C$, obtained via the puncturing operation.
\begin{definition}
    Let $A\in\GL(n,\F)$ and $B\in\GL(m,\F)$. Let $I\subseteq [n]$ and $J\subseteq [m]$ be such that $0<|I|<n$ and $0<|J|<m$. We define the \textbf{row-punctured} code of $\C$ with respect to $A$ and $I$ as
    $$\Pi^r(\C,A,I)=\{(AC)_I\st C\in\C\}\subseteq \Mat(|I|,m,\F),$$
   and the \textbf{column-punctured} code of $\C$ with respect to $B$ and $J$ as
     $$\Pi^c(\C,B,J)=\{(CB)^J\st C\in\C\}\subseteq \Mat(n,|J|,\F).$$
     In other words, $\Pi^r(\C,A,I)$ is the set of all matrices in $A\C$ obtained by deleting the rows not in $I$ and $\Pi^c(\C,B,J)$ is the set of all matrices in $\C B$ obtained by deleting the columns not in~$J$.
\end{definition}

\section{Flag-rank-metric codes}\label{sec:FRM}
In this section we  introduce the notion of \emph{flag-rank-metric codes}, starting from \emph{degenerate flag varieties}. 
In order to define such a degeneration, we follow the notation of \cite{irelli2015degenerate} and \cite{fourier2021degenerate}. Let $\{v_1,\ldots, v_{n+1}\}$ be a basis of $\F^{n+1}$ and consider the $\F$-linear \textbf{projections}
\begin{equation*}
    \mathrm{pr}_{i}:\F^{n+1}\to \F^{n+1}, \; v_j\mapsto \begin{cases}
        v_j, & \textnormal{ if } j\ne i\\
        0, & \textnormal{ if } j= i.
    \end{cases}
\end{equation*}
With respect to $\{v_1,\ldots,v_{n+1}\}$, each $\mathrm{pr}_i$ acts as the right multiplication by the diagonal matrix
$$D_i=\diag(1,1,\ldots, 1,0,1,\ldots, 1)\in\Mat({n+1},\F),$$
having $0$ in the position $(i,i)$.

\noindent
The \textbf{degenerate flag variety} $\Fl^{(a)}(\F^{n+1})$ is the space defined as
$$\Fl^{(a)}(\F^{n+1})=\left\{(V_1,\ldots,V_{n})\in\prod_{i=1}^{n}\Gr_i(\F^{n+1})\st \mathrm{pr}_{i+1}(V_i)\subseteq V_{i+1}, \; \textnormal{for all } i\in[n-1]\right\}.$$
$\Fl^{(a)}(\F^{n+1})$ is a metric space endowed with the component-wise sum-subspace distance $\dd_{\Ss}$, i.e. 
$$\dd_{\Ss}((U_1,\ldots,U_{n}),(V_1,\ldots,V_{n})) = \sum_{i=1}^{n}\dd_{\mathrm{S}}(U_i,V_i).$$
Furthermore, $\Fl^{(a)}(\F^{n+1})$ can be partitioned into cells, where the largest one is 
$$\Fl^{(0a)}(\F^{n+1})=\Fl^{(a)}(\F^{n+1}) \cap \prod_{i=1}^{n}\Gr_i^0(\F^n).$$
Let $(V_1,\ldots,V_{n})\in\Fl^{(0a)}(\F^{n+1})$. For every $i\in[n]$ we can write  
$$V_i=\rs(\; \mathrm{Id}_{i} \, \mid \, A_i \; ), \quad \textnormal{ for a unique } A_i\in\Mat(i,n-i+1,\F).$$
The condition $\mathrm{pr}_{i+1}(V_i)\subseteq V_{i+1}$ translates to 
$$ A_{i+1} = \begin{pmatrix}
    \overline{A}_i \\
    x_{i+1}
\end{pmatrix}, \quad \textnormal{ for some } x_{i+1}\in\F^{n-i},$$
where $ \overline{A}_i\in\Mat(i,n-i,\F)$ is the matrix obtained from $A_i$ by deleting the first column.

 We introduce the following notation which will be helpful for describing a
parametrization of $\Fl^{(0a)}(\F^n)$ via upper triangular matrices.

\begin{notation}
For a given matrix $M \in \Mat(n,\F)$ and an integer $i\in[n]$, we denote by $\MM{M}{i}$ the top-rightmost $i \times (n-i+1)$ submatrix of $M$, that is the one obtained by removing the first $i-1$ columns and the last $n-i$ rows from $M$. Furthermore, if $\C \subseteq \Mat(n,\F)$, we denote by $\MM{\C}{i}$ the set
$$ \MM{\C}{i} \coloneqq \left\{\MM{M}{i} \st M \in \C \right\} \subseteq \Mat(i,n-i+1,\F).$$
\end{notation}

\begin{definition}
    Let $\varphi^{(a)} : \Fl^{(0a)}(\F^{n+1}) \to \U(n,\F)$ be the map given by $\varphi^{(a)}(V_1,\ldots,V_{n})=X$, where
    $X_{[i]}=A_i$ and $\rs(\; \mathrm{Id}_{i}\mid A_i \;)=V_i$. The map $\varphi^{(a)}$ is a bijection from $\Fl^{(0a)}(\F^{n+1})$ to  $\U(n,\F)$ and its inverse map is 
    $$\left(\varphi^{(a)}\right)^{-1}(X) = (V_1,\ldots,V_{n}),$$
    where $V_i=\rs( \; \mathrm{Id}_{i}\mid X_{[i]} \;)$.
\end{definition}

\begin{example}
    Let $n=4$ and let $\F=\mathbb{Q}$ be the field of rational numbers.
    Consider the degenerate flag $(V_1,V_2,V_3,V_4)$ where
    \begin{align*}
        V_1 &=\rs\begin{pmatrix}
            1 & 0 & 1 & 2 & 3 
        \end{pmatrix}, &V_2 =\rs\begin{pmatrix}
            1 & 0 & 1 & 2 & 3 \\
            0 & 1 & 4 & 5 & 6
        \end{pmatrix},\\
        V_3 &=\rs\begin{pmatrix}
            1 & 0 & 0 & 2 & 3 \\
            0 & 1 & 0 & 5 & 6 \\
            0& 0 & 1 & 0 & 8
        \end{pmatrix},  &V_4=\rs\begin{pmatrix}
            1 & 0 & 0 & 0 & 3 \\
            0 & 1 & 0 & 0 & 6 \\
            0& 0 & 1 & 0 & 8 \\
            0 & 0 & 0 & 1 & 0
        \end{pmatrix}.
    \end{align*}
    Note that $(V_1,\ldots,V_4)$ is not a flag since $V_2\not\subseteq V_3$, but it is not difficult to see that for every $i\in\{1,2,3\}$ we have that $\mathrm{pr}_{i+1}(V_i)\subseteq V_{i+1}$.
    Moreover $\varphi^{(a)} : \Fl^{(0a)}(\mathbb{Q}^5) \to \U(4,\mathbb{Q})$ is such that 
    $$\varphi^{(a)}(V_1,V_2,V_3,V_4)=\begin{pmatrix}
        0 & 1 & 2 & 3 \\
        0 & 4 & 5 & 6 \\
        0 & 0 & 0 & 8\\
        0 & 0 & 0 & 0    \end{pmatrix}\in\U(4,\mathbb{Q}).$$
\end{example}

In order to obtain an isometry with respect to the sum-subspace distance $\dd_{\mathrm{SS}}$ we define the flag-rank distance on $\U(n,\F)$ as follows.

\begin{definition}\label{def:frd}
The \textbf{flag-rank weight} of a matrix $M\in \U(n,\F)$ is the quantity
$$\w_{\ff}(M_{[i]}):=\sum_{i=1}^{n}\rk(M_{[i]}).$$
It induces the \textbf{flag-rank distance} on $\U(n,\F)$ defined as the map 
$$ \begin{array}{rccc}\df:&\U(n,\F) \times \U(n,\F)& \longrightarrow & \mathbb N \\ & (A,B) & \longmapsto & \w_{\ff}(A-B).
\end{array}$$
\end{definition}

The following result from \cite{fourier2021degenerate} points out the isometric relation between degenerate flag varieties and upper triangular matrices.

\begin{theorem}[{\cite[Theorem 2.5]{fourier2021degenerate}}]\label{thm:isometry}
    For every $(U_1,\ldots,U_{n}), (V_1,\ldots,V_{n}) \in \Fl^{(0a)}(\F^{n+1})$, we have 
    $$ \dd_{\Ss}((U_1,\ldots,U_{n}), (V_1,\ldots,V_{n})) = 2\cdot\df(\varphi^{(a)}(U_1,\ldots,U_{n}),\varphi^{(a)}(V_1,\ldots,V_{n})),$$
    i.e. the metric spaces $(\Fl^{(0a)}(\F^{n+1}), \dd_{\Ss})$ and $(\U(n,\F),2\cdot\df)$ are isometric. 
\end{theorem}

Note that in \cite[Section 5]{fourier2021degenerate}, the authors highlight the advantages of using degenerate flags instead of flags in network coding. In particular, network coding operations in $\Fl^{(a0)}(\F^{n+1})$ involve less computations than
the ones in $\Fl^{(0)}(\F^{n+1})$ due to the underlying vector space structure of $\U(n,\F)$. As linear network coding motivated the study of rank-metric codes, Theorem~\ref{thm:isometry} motivates the introduction of flag-rank-metric codes, which are defined as follows. 

\begin{definition}
A \textbf{(linear) flag-rank-metric code} is a $t$-dimensional $\F$-subspace $\C$ of $\U(n,\F)$ endowed with the flag-rank distance. The \textbf{minimum flag-rank distance} of $\C$ is the quantity
$$ \delta=\df(\C):=\min\{\wf(M) \st M\in \C, M\neq 0 \}.$$
We will refer to $\C$ as an $\ntdf$ code. 
\end{definition}

\begin{remark}\label{rem:wmax}
It is immediate to see that the maximum possible flag-rank weight that a matrix $M\in\U(n,\F)$ can have is 
$$ \w_{\max}(n) =\begin{cases}
k^2 & \textnormal{ if } n=2k-1, \\
k(k+1) & \textnormal{ if } n=2k.
\end{cases}$$
An $\ntdf$ code with $\delta=\w_{\max}(n)$ is said to be \textbf{optimum-distance}. 
\end{remark}

%%%%%%%%%%%%%%%%%%%%%%%%%%%%%%%%%%%%%%%%%%%%%%%%%%%%%%%%%%%%%%%%%%%%%%%%%%%%%%%%%%%%%%%%%%%%%%%%%%%%
\section{A Singleton-like bound}\label{sec:singletonlikebound}

In this section we prove a Singleton-like bound for linear flag-rank-metric codes. The idea of the proof is inspired by  \cite[Proposition~1]{fourier2021degenerate} and \cite[Proposition~2]{fourier2021degenerate}, where a bound is provided in the case of optimum-distance codes. 

Let $\C$ be an $\ntt{n,t,\delta}$ flag-rank-metric code and let $k=\left\lceil\frac{n}{2}\right\rceil$.
In order to simplify the proof, we will write the parameters of the code as follows, according to whether $n$ is even or odd.

\noindent
If $n=2k$, we consider the unique $s,\ell$ and $j$ such that \[k(k+1)-\delta=\underbrace{k+k+\ldots+(k-\ell+1)}_{s \text{ summands }}+j,\]
where $0\leq j < k-\lfloor \frac{s}2\rfloor$.
If $n=2k-1$, we consider the unique $s,\ell$ and $j$ such that
\[k^2-\delta=\underbrace{k+(k-1)+(k-1)+\ldots+(k-\ell)}_{s \text{ summands }}+j,\]
where $0\leq j < k-\lfloor \frac{s}2\rfloor$.

The main idea for proving the Singleton-like bound for an $\ntt{n,t,\delta}$ code $\C$ relies on the projection of $\C$ on the last $s$ columns and part of the $(n-s)$-th column. 
We first state a more convoluted version of our Singleton-like bound, which will be then useful to deduce a more compact and elegant version afterwards. 

\begin{proposition}\label{pr:Singleton}
Let $\mathcal{C}\subseteq \U(n,\F)$ be a linear flag-rank-metric code with minimum distance $\delta$.
Let $s,\ell$ and $j$ be as above.
Then
\begin{equation*}
    \dim_{\F}(\mathcal{C})\leq n+(n-1)+\ldots+(n-s+1)+\alpha(s),
\end{equation*} 
where
\[ \alpha(s)=\begin{cases}
\frac{n-s}2+j+1, & \text{if $n-s$ is even},\\
\frac{n-s+1}2+j, & \text{if $n-s$ is odd}.
\end{cases}
\]
\end{proposition}
\begin{proof}
Towards a contradiction, suppose that
\[ \dim_{\F}(\mathcal{C})\geq n+(n-1)+\ldots+(n-s+1)+\alpha(s)+1, \]
and consider the projection of $\mathcal{C}$ onto the set of $n+(n-1)+\ldots+(n-s+1)+\alpha(s)$ positions 
$$\mathcal{X}=\{(i,n-s+a) \st a \in [s], \; i \in [n-s+a]\}\cup\{(i,n-s) \st i\in[\alpha(s)]\}.$$
By the assumption on $\dim_{\F}(\C)$, such a projection cannot be injective, and hence there exists a non-zero matrix $M\in\C$ with all zeros in the entries indexed by $\mathcal X$. 
%Note that the last $s$ columns of $M$ are zero. 
Let $\overline{M}$ be the matrix in $\U(n-s,\mathbb{F})$ obtained from $M$ by deleting the last $s$ rows and the last $s$ columns. Since the last $s$ columns of $M$ are zero, then $\mathrm{w}_{\mathrm{fr}}(M)=\mathrm{w}_{\mathrm{fr}}(\overline{M})$. Moreover, the first $\alpha(s)$ entries of the last column of $\overline{M}$ are zero.
Observe also that $j$ can be written as follows
\[ j=\begin{cases}
\frac{n-s}2\left(\frac{n-s}2+1 \right)-\delta, & \text{if $s$ is even},\\
\left(\frac{n-s+1}2 \right)^2-\delta, & \text{if $s$ is odd}.\\
\end{cases}
\]
Now, suppose that $n-s=2 \overline{k}$.
Since the last column of the matrices $\overline{M}_{[\overline{k}+1]},\ldots,\overline{M}_{[\overline{k}+j+1]}$ is the zero column and their number of columns is smaller than their number of rows, it follows that $\rk\left(M_{[\overline{k}+b]}\right)\leq \overline{k}-b$ for all $b \in [j+1]$ (one less than their number of columns).
This implies that
\begin{align*}\mathrm{w}_{\mathrm{fr}}(M)&=\sum_{b=1}^{2\overline{k}}\rk\left(M_{[b]}\right)\leq \sum_{b=1}^{\overline{k}}\rk\left(M_{[b]}\right)+\sum_{b=1}^{j+1}\rk\left(M_{[\overline{k}+b]}\right)+\sum_{b=j+2}^{\overline{k}}\rk\left(M_{[\overline{k}+b]}\right)\\&\leq  \sum_{b=1}^{\overline{k}}b+\sum_{b=1}^{j+1}(k-b)+\sum_{b=j+2}^{\overline{k}}(\overline{k}-b+1)=\overline{k}(\overline{k}+1)-j-1=\delta-1,
\end{align*}
yielding a contradiction to the fact that $\mathcal{C}$ has minimum flag-rank distance $\delta$. 
Similar argument can be performed when $n-s$ is odd.
\end{proof}

When $\delta=\w_{\max}(n)$, Proposition \ref{pr:Singleton} gives the bounds shown in \cite[Proposition 1]{fourier2021degenerate} when $n$ is odd and \cite[Proposition 2]{fourier2021degenerate} when $n$ is even.

We can actually give a second version of Proposition \ref{pr:Singleton}, which involves the minimum codimension of a flag-rank-metric code for the given parameters $n$ and $\delta$. The advantage is that this quantity turns out to depend only on $\delta$ and not on $n$ and it is much easier to compute.

\begin{theorem}[Singleton-like bound]\label{th:Singleton_codimension}
Let $\mathcal{C}\subseteq \U(n,\F)$ be a linear flag-rank-metric code with minimum distance $\delta$. Then 
\begin{equation}\label{eq:sing-bound}
    \mathrm{codim}_{\F}(\C)\geq \delta-1+\lfloor{\sqrt{\delta-1}}\rfloor \cdot \left\lfloor{\frac{\sqrt{4(\delta-1)+1}-1}{2}}\right\rfloor.
\end{equation} 
\end{theorem}
\begin{proof}
Let $k,s,\ell,j$ and $\alpha(s)$ be as in Proposition \ref{pr:Singleton}.
By Proposition \ref{pr:Singleton}, we have that 
$$ \mathrm{codim}_{\F}(\C)\geq \frac{n(n+1)}2-(n+\ldots+(n-s+1)+\alpha(s))=\frac{(n-s)(n-s+1)}2-\alpha(s).$$
In order to prove the statement it is sufficient to verify that 
\begin{equation}\label{eq:toprovebound} \frac{(n-s)(n-s+1)}2-\alpha(s)=\delta-1+\lfloor{\sqrt{\delta-1}}\rfloor \cdot \left\lfloor{\frac{\sqrt{4(\delta-1)+1}-1}{2}}\right\rfloor. \end{equation}
We will only prove the above equality  when $n$ and $s$ are both even since in the remaining cases the computations are similar.
Note that, in this case $0\leq j < k-\ell$ and 
\[ \frac{(n-s)(n-s+1)}2-\alpha(s)= \frac{(n-s)^2}2-j-1=2(k-\ell)^2-j-1. \]
Moreover,
\[\delta=k(k+1)-2(k+(k-1)+\ldots+(k-\ell+1))-j=(k-\ell)(k-\ell+1)-j,\]
from which we get
\[ \lfloor{\sqrt{\delta-1}}\rfloor=k-\ell \]
and
\[ \left\lfloor{\frac{\sqrt{4(\delta-1)+1}-1}{2}}\right\rfloor=k-\ell-1, \]
by making use of the restrictions on $j$.
Therefore, \eqref{eq:toprovebound} is verified. 
\end{proof}

From now on, we will define the two following functions coming from Theorem \ref{th:Singleton_codimension}, which give a more compact notation for the Singleton-like bound.
\begin{align}
    g(\delta)&:=\delta-1+\lfloor{\sqrt{\delta-1}}\rfloor \cdot \left\lfloor{\frac{\sqrt{4(\delta-1)+1}-1}{2}}\right\rfloor \\
    f(n,\delta)&:=\frac{n(n+1)}{2}-g(\delta).
\end{align}

%%%%%%%%%%%%%%%%%%%%%%%%%%%%%%%%%%%%%%%%%%%%%%%%%%%%%%%%%%%%%%%%%%%%%%%%%%%%%%%%%%%%%%%%%%%%%%%%%%%%

\section{Maximum flag-rank distance codes}\label{sec:MFRC}

In this section we focus on the study and constructions of extremal codes with respect to the bound of Theorem \ref{th:Singleton_codimension}. 

\begin{definition}\label{def:MFRD}
A flag-rank-metric code is said to be a \textbf{maximum flag-rank distance (MFRD) code} if its parameters meet the bound of Theorem \ref{th:Singleton_codimension} with equality. In other words, a flag-rank-metric code is MFRD if it is an $\ntt{n,f(n,\delta),\delta}$ code, for some $\delta \geq 1$.
\end{definition}
In the rest of this paper we provide a number of different constructions of $\ntt{n,f(n,\delta),\delta}$ MFRD codes, for some values of the minimum flag-rank distance $\delta$.

%%%%%%%%%%%%%%%%%%%%%%%%%%%%%%%%%%%%%%%%%%%%%%%%%%%%%%%%%%%%%%%%%%%%%%%%%%%%%%%%%%%%%%%%%%%%%%%%%%%%

\subsection{Optimum-distance MFRD codes}
In this subsection we focus on optimum-distance MFRD codes, i.e. MFRD codes with $\delta=\w_{\max}(n)$; see Remark \ref{rem:wmax}. 
We present a general construction that works in both the cases $n$ odd and $n$ even. When $n$ is odd, the same construction has been provided in \cite{fourier2021degenerate}, but we repeat it for the the sake of completeness. We recall that $k=\lceil \frac{n}{2}\rceil$, i.e. we have $n=2k$ when $n$ is even and $n=2k-1$ when $n$ is odd.

The following two auxiliary results are straightforward and hence we omit their proofs. 

\begin{lemma}\label{lem:rank-1}
 For every $M\in \U(n,\F)$ and every $i \in [n-1]$ we have
 $$ |\rk(\MM{M}{i+1})-\rk(\MM{M}{i})| \leq 1.$$
\end{lemma}

As an immediate consequence of Lemma \ref{lem:rank-1} we get the next proposition.

\begin{proposition}\label{prop:OMFRDcriterion}
Let $\C\subseteq\U(n,\F)$. The following hold.
\begin{enumerate}
    \item If $n=2k-1$ then $\df(\C)=k^2$ if and only if $\dr(\MM{\C}{k})=k$.
    \item If $n=2k$, then $\df(\C)=k(k+1)$ if and only if $\dr(\MM{\C}{k})=\dr(\MM{\C}{k+1})=k$.
\end{enumerate}
\end{proposition}

In case of optimum distance flag-rank-metric codes the bound of Theorem \ref{th:Singleton_codimension} reads as 
\begin{align*}
    f(2k-1, k^2) &= k, \\
    f(2k, k(k+1)) &= k+1,
\end{align*}

We point out that Fourier and Nebe  in \cite{fourier2021degenerate}  gave a construction of optimum-distance MFRD codes only for $n=2k-1$. Moreover, in the same paper they also dealt with  optimum-distance flag-rank-metric codes with $n=2k$. However, in this case, 
they only construct a $\ntt{2k,k,k(k+1)}$ code, that is one dimension less than an MFRD code.
Here, we provide a construction of  optimum-distance MFRD codes which coincides with the one  in \cite[Proposition 1]{fourier2021degenerate} for $n$ odd and improves the one in \cite[Proposition 2]{fourier2021degenerate} when $n$ is even. Before that, we introduce the following notation, which we illustrate in Example \ref{ex:notation4.4}.

\begin{notation}
    We denote by $\pi$ the projection map from $\Mat(n,\F)$ onto $\U(n,\F)$, that is 
$$\begin{array}{rccl}\pi:&\Mat(n,\F) & \longrightarrow & \U(n,\F) \\
& (a_{ij}) & \longmapsto & (b_{ij})\end{array}$$
where 
$$ b_{ij}=\begin{cases} a_{ij} & \mbox{ if } i \le j, \\
0 & \mbox{ otherwise. }\end{cases}$$
Moreover, let $n,\ell, i,j$ be positive integers such that $i+\ell-1,j+\ell-1 \leq n$. We define the  map 
    $$\phi_{i,j}\colon \Mat(\ell,\F)\rightarrow \Mat(n,\F),$$
    given by $(\phi_{i,j}(M))_I^J=M$, where $I=\{i,i+1,\ldots,i+\ell-1\}$ and $J=\{j,j+1,\ldots,j+\ell-1\}$, and with all the remaining entries equal to zero.  
\end{notation}

\begin{example}\label{ex:notation4.4}
Let $\ell=3$ and $n=5$. The map $\phi_{1,2}:\Mat(3,\mathbb{Q})\longrightarrow \Mat(5,\mathbb{Q})$ acts as
\bigskip 

\[
  M = \left(\begin{array}{*5{c}}
    \tikzmark{left}{5} & 2 & 1 \\
    6 & 7 & 8  \\
    1 & 12 & \tikzmark{right}{5}  
  \end{array}\right)
  \Highlight[first]
  \qquad
  \phi_{1,2}(M) = \left(\begin{array}{*5{c}}
    0 & \tikzmark{left}{5} & 2 & 1 & 0 \\
    0 & 6 & 7 & 8 & 0 \\
   0 & 1 & 12 & \tikzmark{right}{5} & 0 \\
    0 & 0 & 0 & 0 & 0 \\
    0 & 0 & 0 & 0 & 0
  \end{array}\right).
\]
\Highlight[second]
\tikz[overlay,remember picture] {
  \draw[->,thick,red,dashed] (first) -- (second) node [pos=0.66,above] {$\phi_{1,2}$};
  %\node[above of=first] {$N$};
  \node[above of=second] {$M$};
}
\end{example}

\begin{construction}\label{constr:A}
 Let $\epsilon\in\{0,1\}$ and $k \in \mathbb{N}$.
 Let $\D$ be a $[(k+\epsilon)\times (k+\epsilon),(k+\epsilon),(k+\epsilon)]_{\F}$ MRD code. Define the $\ntt{2k-1+\epsilon,k+\epsilon}$ code
 $$\C_{\D}:=\pi\circ\phi_{1,k}(\D).$$
\end{construction}
Figures \ref{picA} and \ref{picAeven} show how the codewords of $\C_{\D}$ look like in case $n$ is odd and $n$ is even, respectively.

\begin{figure}[ht!]
\centering
\begin{tikzpicture}[thick, scale=0.6]
\draw[help lines, very thick] (0,0) -- (1,0)-- (1,-1) -- (2,-1) -- (2,-2)-- (3,-2) -- (3,-3)-- (4,-3) -- (4,-4)--(5,-4)--(5,1) -- (0,1)--(0,0);
\draw[red, very thick] (2,1)-|(2,-2)-|(5,-2)-| (5,1) -|cycle;
\node[red] at (3.5,-0.5) {$\D$};
\end{tikzpicture}
\caption{Representation of the codewords from Construction \ref{constr:A} when $n$ is odd. In the red box we embed the codewords from $\D$, while outside we only have $0$ entries.}
\label{picA}
\end{figure}
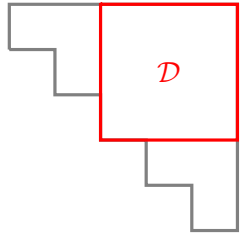

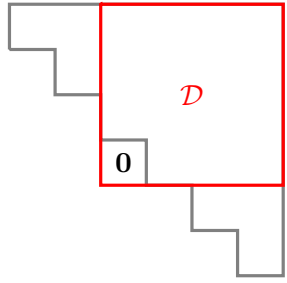
\begin{figure}[ht!]
\centering
\begin{tikzpicture}[thick, scale=0.6]
\draw[help lines, very thick] (0,0) -- (1,0)-- (1,-1) -- (2,-1) -- (2,-2)-- (3,-2) -- (3,-3)-- (4,-3) -- (4,-4)--(5,-4)--(5,-5) --(6,-5) -- (6,1)-- (0,1)--(0,0);
\draw[red, very thick] (2,1)-|(2,-3)-|(6,-3)-| (6,1) -|cycle;
\node[red] at (4,-1) {$\D$};
\node at (2.5,-2.5) {$\boldsymbol{0}$};
\end{tikzpicture}
\caption{Representation of the codewords from Construction \ref{constr:A} when $n$ is even. In the red box there are the projections of the codewords from $\D$, while outside we only have $0$ entries.}
\label{picAeven}
\end{figure}

\begin{theorem}\label{th:extopt}
 For $\epsilon \in \{0,1\}$, $k\in\mathbb N$ and every $[(k+\epsilon)\times (k+\epsilon),(k+\epsilon),(k+\epsilon)]_{\F}$ MRD code $\D$, the code $\C_\D:=\pi\circ\phi_{1,k}(\D)$ defined in Construction \ref{constr:A} is a $\ntt{2k-1+\epsilon,k+\epsilon, k(k+\epsilon)}$ optimum-distance MFRD code.
\end{theorem}

\begin{proof}
 The case $n=2k-1$ is shown in \cite[Proposition 1]{fourier2021degenerate}. Hence, we only prove the claim when $n=2k$. By Proposition \ref{prop:OMFRDcriterion}(2), we have to show that $\dr(\MM{(\C_\D)}{k})=\dr(\MM{(\C_\D)}{k+1})=k$. 
 It is immediate to verify that, for $n=2k$, we have 
 \begin{align*} \MM{(\pi\circ\phi_{1,k}(\D))}{k}&=\Pi^r(\D,\mathrm{Id}_{k+1},[k]) \\
 \MM{(\pi\circ\phi_{1,k}(\D))}{k+1}&=\Pi^c(\D,\mathrm{Id}_{k+1},[k]).
 \end{align*}
 Since the puncturing of a $[(k+1)\times (k+1),(k+1),(k+1)]_{\F}$ MRD code on the first $k$ rows (respectively columns) is a $[k \times (k+1),(k+1),k]_{\F}$ MRD code (respectively a  $[(k+1) \times k,(k+1),k]_{\F}$ MRD code), this concludes the proof.
\end{proof}

\begin{remark}
When $\mathbb{F}$ is finite, the MRD codes involved in Construction \ref{constr:A} correspond to \emph{semifields}. A \textbf{finite semifield} $(\mathbb{S},+,\circ)$ is a finite non-associative division ring with multiplicative identity element. If $\mathbb{S}$ has finite dimension $n$ over a subfield $\mathbb{K}$, then consider the set $\mathcal{R}$ of $\mathbb{K}$-linear maps defined by the right multiplication of the elements in $\mathbb{S}$ (also known as the \textbf{spread set} of $\mathbb{S}$), that is the elements of $\mathcal{R}$ are the $\mathbb{K}$-linear maps
$$\begin{array}{rccl}M_a:&\mathbb{S} & \longrightarrow &\mathbb{S} \\
& b & \longmapsto & a \circ b,\end{array}$$
% \[ M_a \colon b \in \mathbb{S} \mapsto a \circ b\in \mathbb{S}, \]
for every $a \in \mathbb{S}$. Since $\dim_{\mathbb{K}}(\mathbb{S})=n$ and $\mathcal{R}$ results to be a $\mathbb{K}$-vector space, then $\mathcal{R}$ can be seen as an $[n\times n,n,n]_{\mathbb{K}}$ code; see e.g. \cite{lavrauw2011finite,sheekey2020new}.
It is well-known that there is a one-to-one correspondence between isotopy classes of finite semifields of dimension $n$ over $\mathbb{K}$ with equivalence classes of $[n\times n, n,n]_{\mathbb{K}}$ MRD codes; see \cite{de2016algebraic}.
\end{remark}

The next example illustrates Construction \ref{constr:A} when $n$ is even.

\begin{example}\label{ex:ConstA}
Let $n=4$. Define $\D\subseteq \Mat(3,\F_2)$ to be the $[3\times 3,3,3]_{\F_2}$ MRD code generated by the following matrices
$$ D_1=\begin{pmatrix}
1 & 0 & 0 \\
0 & 1 & 0 \\
0 & 0 & 1
\end{pmatrix}, \quad D_2=\begin{pmatrix}
0 & 0 & 1 \\
1 & 0 & 1 \\
0 & 1 & 0
\end{pmatrix}, \quad D_3=\begin{pmatrix}
0 & 1 & 0 \\
0 & 1 & 1 \\
1 & 0 & 1
\end{pmatrix}.$$
Now, we have that $\phi_{1,2}(\D)$ is the subspace of $\Mat(4,\F_2)$ generated by the matrices
$$ \phi_{1,2}(D_1)=\begin{pmatrix}
0 & 1 & 0 & 0 \\
0 & 0 & 1 & 0 \\
0 & 0 & 0 & 1 \\
0 & 0 & 0 & 0
\end{pmatrix}, \quad \phi_{1,2}(D_2)=\begin{pmatrix}
0 & 0 & 0 & 1 \\
0 & 1 & 0 & 1 \\
0 & 0 & 1 & 0 \\
0 & 0 & 0 & 0
\end{pmatrix}, \quad \phi_{1,2}(D_3)=\begin{pmatrix}
0 & 0 & 1 & 0 \\
0 & 0 & 1 & 1 \\
0 & 1 & 0 & 1 \\
0 & 0 & 0 & 0
\end{pmatrix}.$$
By applying $\pi$ to $\phi_{1,2}(\D)$, we get that all the entries below the main diagonal are zeros, hence, the code $\C_\D=\pi\circ\phi_{1,2}(\D)\subseteq \U(4,\F_2)$ is the $\{4,3,6\}_{\F_2}$ MFRD code generated by 
$$ M_1=\begin{pmatrix}
0 & 1 & 0 & 0 \\
0 & 0 & 1 & 0 \\
0 & 0 & 0 & 1 \\
0 & 0 & 0 & 0
\end{pmatrix}, \quad M_2=\begin{pmatrix}
0 & 0 & 0 & 1 \\
0 & 1 & 0 & 1 \\
0 & 0 & 1 & 0 \\
0 & 0 & 0 & 0
\end{pmatrix}, \quad  M_3=\begin{pmatrix}
0 & 0 & 1 & 0 \\
0 & 0 & 1 & 1 \\
0 & 0 & 0 & 1 \\
0 & 0 & 0 & 0
\end{pmatrix}.$$
\end{example}

\begin{remark}[Nonexistence of optimum distance MFRD codes]\label{rem:nonexistence}
Observe that when $n=2k-1$, Construction \ref{constr:A} is essentially the only possible -- up to arbitrarily changing everything that is not in the $k\times k$ top-rightmost block. This is due to Proposition \ref{prop:OMFRDcriterion}.

Furthermore, the same result  implies that  when $n=2k-1$, an optimum distance MFRD code over $\F$ exists if and only there exists a $[k\times k,k,k]_{\F}$ MRD code. As already pointed out, it is well-known that MRD codes with these parameters exist over any finite field, but they do not exist over algebraically closed fields, unless $k=1$. 

For the case of optimum distance MFRD codes for $n=2k$, Proposition \ref{prop:OMFRDcriterion} still gives  a necessary (but maybe not sufficient) condition for their existence, that is the existence of a $[k\times(k+1),k+1,k]_{\F}$ MRD code. Again, these codes do not exist over algebraically closed fields, unless $k=1$.
\end{remark}

\subsection{Quasi-optimum-distance MFRD codes}
In this subsection we provide a construction of \textbf{quasi-optimum MFRD} codes, i.e. MFRD with minimum flag-rank distance~$\delta=\w_{\max}(n)-1$. However, our construction is only valid for $n$ odd. Note that this construction is similar to  Construction \ref{constr:A}.

\begin{construction}\label{constr:B}
 Let $k\in\mathbb{N}$ and let $\D$ be a $[(k+1)\times (k+1),(k+1),(k+1)]_{\F}$ MRD code. Define the $\ntt{2k-1,k+1}$ code
 $$\C_{\D}:=\pi\circ\phi_{1,k-1}(\D).$$
\end{construction}

In Figure \ref{picA'} we represent the codewords of a code obtained via Construction \ref{constr:B}.

\begin{figure}[ht!]
\centering
\begin{tikzpicture}[thick, scale=0.6]
\draw[help lines, very thick] (0,0) -- (1,0)-- (1,-1) -- (2,-1) -- (2,-2)-- (3,-2) -- (3,-3)-- (4,-3) -- (4,-4)--(5,-4)--(5,1) -- (0,1)--(0,0);
\draw[red, very thick] (1,1)-|(1,-3)-|(5,-3)-| (5,1) -|cycle;
\node at (2.5,-2.5) {$\boldsymbol{0}$};
\node at (1.5,-1.5) {$\boldsymbol{0}$};
\node at (1.5,-2.5) {$\boldsymbol{0}$};
\node[red] at (3,-0.5) {$\D$};
\end{tikzpicture}
\caption{Representation of the codewords from Construction \ref{constr:B}. In the red box there are the projections of the codewords from $\D$.}
\label{picA'}
\end{figure}

\begin{theorem}\label{th:existnoddmrd}
 For every  $k\in\mathbb N$ and every $[(k+1)\times (k+1),(k+1),(k+1)]_{\F}$ MRD code $\D$, the code $\C_\D:=\pi\circ\phi_{1,k-1}(\D)$ defined in Construction \ref{constr:B} is a $\{2k-1,k+1,k^2-1\}_{\F}$ MFRD code.
\end{theorem}

\begin{proof}
Note that $(\pi\circ\phi_{1,k-1}(\D))_{[k]}=\Pi^r(\mathcal{D},\mathrm{Id}_{k+1},[k])$ and since $\mathcal{D}$ is an MRD code with minimum distance $k+1$, then
\[ \mathrm{d_r}((\pi\circ\phi_{1,k-1}(\D))_{[k]})\geq k-1, \]
and 
\[ \mathrm{d_r}((\pi\circ\phi_{1,k-1}(\D))_{[i]})= \min\{ i,n-i+1 \}, \quad \text{ for }\,\, i \ne k. \]
The assertion then follows from Theorem \ref{th:Singleton_codimension}.
\end{proof}

\begin{example}\label{constrB}
    Let $k=3$ and $n=2k-1=5$. Consider the $[4\times 4,4,4]_{\F_2}$ MRD code $\D\subseteq \Mat(4,\F_2)$ generated by the following matrices
    $$ D_1=\begin{pmatrix}
        1 & 0 & 0 & 0\\
        0 & 1 & 0 & 0\\
        0 & 0 & 1 & 0\\
        0 & 0 & 0 & 1
    \end{pmatrix}, \;\; D_2=\begin{pmatrix}
        0 & 0 & 0 & 1\\
        1 & 0 & 0 & 1\\
        0 & 1 & 0 & 0\\
        0 & 0 & 1 & 0
    \end{pmatrix}, \;\; D_3=\begin{pmatrix}
        0 & 0 & 1 & 0\\
        0 & 0 & 1 & 1\\
        1 & 0 & 0 & 1\\
        0 & 1 & 0 & 0
    \end{pmatrix}, \;\; D_4=\begin{pmatrix}
        0 & 1 & 0 & 0\\
        0 & 1 & 1 & 0\\
        0 & 0 & 1 & 1\\
        1 & 0 & 0 & 1
    \end{pmatrix}.$$
    Consider now $\phi_{1,2}(\D)\subseteq \Mat(5,\F_2)$, which is the space generated by the matrices 
   \begin{align*}
   \phi_{1,2}(D_1)&=\begin{pmatrix}
        0 & 1 & 0 & 0 & 0\\
        0 & 0 & 1 & 0 & 0\\
        0 & 0 & 0 & 1 & 0\\
        0 & 0 & 0 & 0 & 1\\ 
        0 & 0 & 0 & 0 & 0
    \end{pmatrix}, \; \phi_{1,2}(D_2)=\begin{pmatrix}
        0 & 0 & 0 & 0 & 1\\
        0 & 1 & 0 & 0 & 1\\
        0 & 0 & 1 & 0 & 0\\
        0 & 0 & 0 & 1 & 0\\
        0 & 0 & 0 & 0 & 0
    \end{pmatrix},\\
    \phi_{1,2}(D_3)&=\begin{pmatrix}
        0 & 0 & 0 & 1 & 0\\
        0 & 0 & 0 & 1 & 1\\
        0 & 1 & 0 & 0 & 1\\
        0 & 0 & 1 & 0 & 0\\
        0 & 0 & 0 & 0 & 0
    \end{pmatrix}, \; \phi_{1,2}(D_4)=\begin{pmatrix}
        0 & 0 & 1 & 0 & 0\\
        0 & 0 & 1 & 1 & 0\\
        0 & 0 & 0 & 1 & 1\\
        0 & 1 & 0 & 0 & 1\\
        0 & 0 & 0 & 0 & 0
    \end{pmatrix}.
    \end{align*}
    By applying $\pi$ to $\phi_{1,2}(\D)$, we get that all the entries below the main diagonal are zeros, hence, the code $\C_\D=\pi\circ\phi_{1,2}(\D)\subseteq \U(5,\F_2)$ is the $\{5,4,8\}_{\F_2}$ MFRD code generated by 
       \begin{align*}
   M_1&=\begin{pmatrix}
        0 & 1 & 0 & 0 & 0\\
        0 & 0 & 1 & 0 & 0\\
        0 & 0 & 0 & 1 & 0\\
        0 & 0 & 0 & 0 & 1\\ 
        0 & 0 & 0 & 0 & 0
    \end{pmatrix}, \; M_2=\begin{pmatrix}
        0 & 0 & 0 & 0 & 1\\
        0 & 1 & 0 & 0 & 1\\
        0 & 0 & 1 & 0 & 0\\
        0 & 0 & 0 & 1 & 0\\
        0 & 0 & 0 & 0 & 0
    \end{pmatrix},\\
    M_3&=\begin{pmatrix}
        0 & 0 & 0 & 1 & 0\\
        0 & 0 & 0 & 1 & 1\\
        0 & 0 & 0 & 0 & 1\\
        0 & 0 & 0 & 0 & 0\\
        0 & 0 & 0 & 0 & 0
    \end{pmatrix}, \; M_3=\begin{pmatrix}
        0 & 0 & 1 & 0 & 0\\
        0 & 0 & 1 & 1 & 0\\
        0 & 0 & 0 & 1 & 1\\
        0 & 0 & 0 & 0 & 1\\
        0 & 0 & 0 & 0 & 0
    \end{pmatrix}.
    \end{align*}
\end{example}

\subsection{MFRD codes with small minimum flag-rank  distance}

In this subsection we focus on constructions of MFRD codes with small minimum distance $\delta \leq n$. 

Of course, the only trivial MFRD code with $\delta=1$ is the whole space $\U(n,\F)$. Thus, the only meaningful cases are when $\delta \geq 2$. 

We are now going to prove a very useful result, showing that, when $\delta \le n$, we can restrict ourselves to study MFRD codes containing copies of smaller upper triangular matrix spaces. For an integer $0\leq r \leq n$, let us denote by $\mathrm{D}(r,n,\F)$ the set of upper triangular matrices with only nonzero entries in the first $r$ diagonals, that is
$$ \mathrm{D}(r,n,\F):=\left\{ M=(m_{ij})\in \U(n,\F) \, :\, m_{ij}=0 \mbox{ if } j-i\ge r\right\}.$$

We first need  the following auxiliary result.

\begin{lemma}\label{lem:low_weight_diagonal}
 Let $M\in \U(n,\F)$ be such that $\wf(M)=r\leq n$. Then, $M\in \mathrm{D}(r,n,\F)$.
\end{lemma}

\begin{proof}
Since the zero matrix belongs to $\mathrm{D}(r,n,\F)$, we can assume $M$ to be nonzero. If the matrix $M$ has a nonzero entry which is not in the first $r$ diagonals, then this entry appears in more than $r$  submatrices $M_{[i]}$'s, yielding each time rank at least one. Hence the flag-rank weight of $M$ is at least $r+1$.
\end{proof}

\begin{theorem}\label{thm:diagonal_projection}
Let $\delta \in [n]$, and let $\C$ be an $\ntt{n,f(n,\delta),r}$ code. The following are equivalent:
\begin{enumerate}
    \item $r=\delta$, that is, $\C$ is an $\ntt{n,f(n,\delta),\delta}$ MFRD code.
    \item $\C\cap \mathrm{D}(\delta-1,n,\F)$ is an $\ntt{n,y,\geq \delta}$ code whose codimension in $\mathrm{D}(\delta-1,n,\F)$ is $g(\delta)$.
    %coincides with the codimension of $\C$ in $\U(n,\F)$.
    In other words,
    $$ y=f(n,\delta)-\frac{(n-\delta+1)(n-\delta+2)}{2}.$$
    \item $(\C\cap \mathrm{D}(\delta-1,n,\F)) \oplus (\phi_{1,\delta}(\U(n-\delta+1,\F)))$ is an $\ntt{n,f(n,\delta),\delta}$ MFRD code.
\end{enumerate}
\end{theorem}

\begin{proof}
(2) $\Longleftrightarrow$ (3): By Lemma \ref{lem:low_weight_diagonal}, adding elements not in $\mathrm{D}(\delta-1,n,\F)$ cannot bring the minimum flag-rank distance of a code below $\delta$.

(1) $\Longrightarrow$ (2): Let $g(\delta)$ be the codimension of $\C$. Assume on the contrary that $\C\cap \mathrm{D}(\delta-1,n,\F)$ has codimension $g(\delta)-\epsilon$ in $\mathrm{D}(\delta-1,n,\F)$, with $\epsilon \ge 1$. Then the code $\C':=(\C\cap\mathrm{D}(\delta-1,n,\F))\oplus (\phi_{1,\delta}(\U(n-\delta+1,\F)))$ has codimension $g(\delta)-\epsilon<g(\delta)$ in $\U(n,\F)$. Moreover, by Lemma \ref{lem:low_weight_diagonal} every matrix with a nonzero entry in the last $n-\delta+1$ upper diagonal has flag-rank weight at least $\delta$, and hence we can deduce that the minimum flag-rank distance of $\C'$ is at least $\delta$. However, this contradicts Theorem \ref{th:Singleton_codimension}.

(2) $\Longrightarrow$ (1): By contradiction, assume that $\C$ is not MFRD. Then there exists a nonzero  $M\in\C$ with $\wf(M)=r\le\delta-1$. By Lemma \ref{lem:low_weight_diagonal} this implies that $M\in\mathrm{D}(\delta-1,n,\F)$, and thus the minimum flag-rank distance of $\C\cap \mathrm{D}(\delta-1,n,\F)$ is strictly smaller than $\delta$, contradicting the assumption.
\end{proof}

A direct consequence of Theorem \ref{thm:diagonal_projection} is that, for $\delta \leq n$, an $\ntt{n,f(n,\delta),\delta}$ MFRD code exists if and only if there exists an MFRD code with the same parameters of the form 
$$ \C'=\mathcal D \oplus (\phi_{1,\delta}(\U(n-\delta+1,\F))).$$
where $\mathcal D\subseteq \mathrm{D}(\delta-1,n,\F)$. In particular, in the next constructions we will focus on constructing the summand $\mathcal D\subseteq \mathrm{D}(\delta-1,n,\F)$, which must be an $\ntt{n,f(n,\delta)-\frac{(n-\delta+1)(n-\delta+2)}{2},\geq \delta}$ code.

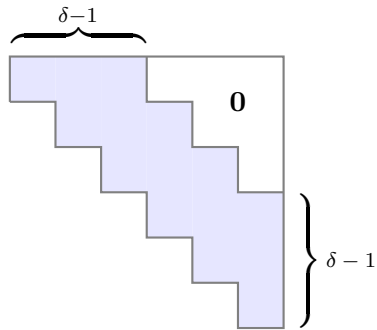
\begin{figure}[ht!]
\centering
\begin{tikzpicture}[thick, scale=0.6]

\fill[blue!10] (0,1) -| (0,0) -| (1,0) -| (1,1) -| cycle;
%\draw[fill=blue!10] (0,1) -| (0,0) -| (1,0) -| (1,1) -| cycle;
\fill[blue!10] (1,0) -| (1,-1) -| (2,-1) -| (2,1) -| cycle;
\fill[blue!10] (2,-1) -| (2,-2) -| (3,-2) -| (3,1) -| cycle;
\fill[blue!10] (3,-2) -| (3,-3) -| (4,-3) -| (4,0) -| cycle;
\fill[blue!10] (4,-3) -| (4,-4) -| (5,-4) -| (5,-1) -| cycle;
\fill[blue!10] (5,-4) -| (5,-5) -| (6,-5) -| (6,-2) -| cycle;
\node at (5,0) {$\boldsymbol{0}$};
\node at (1.5,1.7) {$\overbrace{\qquad \qquad \;\;}^{\delta-1}$};
\node at (7.2,-3.5) {\rotatebox[origin=c]{-90}{$\overbrace{\qquad \qquad \;\;}^{\text{\rotatebox[origin=c]{90}{\;$\delta-1$}}}$}};

\draw[help lines, thick] (0,0) -- (1,0)-- (1,-1) -- (2,-1) -- (2,-2)-- (3,-2) -- (3,-3)-- (4,-3) -- (4,-4)--(5,-4) -- (5,-5) -- (6,-5) -- (6,1) -- (0,1)--(0,0);
\draw[help lines, thick] (3,1) --  (3,0) -- (4,0) -- (4,-1) -- (5,-1) -- (5,-2) -- (6, -2);
\end{tikzpicture}
\caption{The blue area in the picture represents the free entries of the matrices in $\mathrm{D}(\delta-1,n,\F)$.}
\label{pic_proj}
\end{figure}

Using a similar idea to the one used for Theorem \ref{thm:diagonal_projection}, we can relate the existence of MFRD codes of small minimum flag-rank distance with the existence of certain Hamming-metric codes.

\begin{proposition}\label{prop:flag_to_Hamming}
The spaces $(\F^n,\dd_{\mathrm{H}})$ and   $(\mathrm{D}(1,n,\F),\dd_{\ff})$ are naturally isometric. Moreover, if $\C$ is an $\ntt{n,t,\delta}$ code, then $\C\cap \mathrm{D}(1,n,\F)$ is an $[n, \geq n-\frac{n(n+1)}{2}+t,\geq \delta]_{\F}$ code endowed with the Hamming metric.
\end{proposition}
\begin{proof}
The first part of the statement is a direct checking that the map $\psi:\F^n\longrightarrow \mathrm{D}(1,n,\F)$ mapping a vector $v$ in the diagonal matrix whose diagonal is $v$ is an isometry.
Now, let $\overline{\C}:=\C\cap \mathrm{D}(1,n,\F)$ considered as an Hamming-metric code. Then, its codimension in $\mathrm{D}(1,n,\F)$ cannot be larger than the codimension of $\C$ in $\U(n,\F)$. Concerning its minimum Hamming distance, the fact that it must be at least $\delta$ follows from the isometry between $(\mathrm{D}(1,n,\F),\dd_{\ff})$ and $(\F^n,\dd_{\mathrm{H}})$.
\end{proof}

We can now deduce from Proposition \ref{prop:flag_to_Hamming} that MFRD codes cannot exist for every set of parameters over any finite field, in contrast to MRD codes.

\begin{corollary}\label{cor:nonexistence}
Let $\delta \geq 2$ be such that 
$2\leq g(\delta)\leq n-1$ and let $$m(\delta):=\max\{ m \,:\, \mbox{there exists an } [m,m-g(\delta),\delta]_{\F} \mbox{ code } \}.$$ 
If $n>m(\delta)$, then there is no $\ntt{n,f(n,\delta),\delta}$ MFRD code. In particular, when $\F$ is finite, if an $\ntt{n,f(n,\delta),\delta}$ MFRD code exists, then 
$$n\leq(\delta-2)\frac{|\F|^{g(\delta)}-1}{|\F|^{g(\delta)-1}-1}.$$
\end{corollary}

\begin{proof}
 By Proposition \ref{prop:flag_to_Hamming}, if an $\ntt{n,f(n,\delta),\delta}$ code exists, then $\overline{\C}:=\C\cap \mathrm{D}(1,n,\F)$ is an $[n,s, \geq \delta]_{\F}$ code with $n-2\geq s \geq n-g(\delta)\ge 1$. By looking at the columns of a parity-check matrix for $\overline{\C}$ as a set $\mathcal M$ of $n$ points in 
 $\PG(n-s-1,\F)$, we have that  no $\delta-1$ of them are contained in the same $(n-s-2)$-dimensional projective subspace  (that is, a hyperplane).
 Using this fact and a double-counting argument we get,
 $$ n|\PG(n-s-2,\F)|=\sum_{\Lambda \in \mathcal{H}} |\Lambda \cap \mathcal M| \leq (\delta-2)|\PG(n-s-1,\F)|, $$
 where $\mathcal{H}:=\{ \Lambda \subseteq \PG(n-s-1,\F) \st \dim(\Lambda)=n-s-2 \}$. From the above equality we derive
 $$ n \leq (\delta-2)\frac{|\PG(n-s-1,\F)|}{|\PG(n-s-2,\F)|}=(\delta-2)\frac{|\F|^{n-s}-1}{|\F|^{n-s-1}-1}\leq (\delta-2)\frac{|\F|^{g(\delta)}-1}{|\F|^{g(\delta)-1}-1}.$$
\end{proof}

%%%%%%%%%%%%%%%%%%%%%%%%%%%%%%%%%%%%%%%%%%%%%%%%%%%%%%%%%%%%%%%%%%%%%%%%%%%%%%%%%%%%%%%%%%%%%%%%%%%%%

\subsubsection{MFRD codes with $\delta=2$}\label{sec:d=2}

When $\delta=2$, an MFRD code $\C$ has codimension $g(\delta)=1$ and hence it is an  $\ntt{n,\frac{n(n+1)}{2}-1,2}$ code. By Theorem \ref{thm:diagonal_projection}, we have that $\C\cap \mathrm{D}(1,n,\F)$ is an $\ntt{n,n-1,\geq 2}$ code, that is an $[n,n-1,\geq 2]_\F$ code in the Hamming metric by Proposition \ref{prop:flag_to_Hamming}. Thus, MFRD codes with $\delta=2$ exist if and only if there exist an $[n,n-1,2]_\F$ code in the Hamming metric. Such a code $\C_1$ exists over any field $\F$: we can take for instance 
$$\C_1=(1,1,\ldots,1)^\perp.$$
Thus, using Theorem \ref{thm:diagonal_projection}, the $\ntt{n,\frac{n(n+1)}{2}-1}$ code 
$$\overline{\C}_1\oplus\overline{\U}(n-1,\F)$$
is an $\ntt{n,\frac{n(n+1)}{2}-1,2}$ code over any field $\F$, where  $\overline{\C}_1$ is the image of $\C_1$ under the natural mapping $\psi:\F^n \rightarrow \mathrm{D}(1,n,\F)$ and $\overline{\U}(n-1,\F)=\phi_{1,2}(\U(n-1,\F))$.

%%%%%%%%%%%%%%%%%%%%%%%%%%%%%%%%%%%%%%%%%%%%%%%%%%%%%%%%%%%%%%%%%%%%%%%%%%%%%%%%%%%%%%%%%%%%%%%%%%%%

\subsubsection{MFRD codes with $\delta=3$}

Things are much more interesting when studying the case $\delta=3$. Also in this case, we can completely characterize the parameters for which an $\ntt{n,f(n,3),3}$ MFRD code exists. We will do it by considering the intersection with $\mathrm{D}(2,n,\F)$, as derived in Theorem \ref{thm:diagonal_projection}. However, this is not enough. We will then study which matrices in $\mathrm{D}(2,n,\F)$ have flag-rank smaller than $3$, and characterize codes avoiding these matrices. This will be done by using a geometric argument on subspaces of $\mathrm{D}(2,n,\F)$ endowed with the Hamming metric.

Let us introduce the following notation. For $i\in[n]$, let $E_i$ be the matrix whose $(i,i)$-entry is~$1$ and all the other entries are $0$ and, for $i\in[n-1]$, let $D_i$ be the matrix whose $(i,i+1)$-entry is~$1$ and all the other entries are $0$. So, the set 
$$ \{E_i \,:\, i \in [n]\} \cup \{D_i \,:\, i \in [n-1]\}$$
is a basis of $\mathrm{D}(2,n,\F)$. Let us consider the vectorization isomorphism 
$$ \begin{array}{rccc} \varrho:& \mathrm{D}(2,n,\F) & \longrightarrow & \F^{2n-1} \\
&\sum\limits_{i=1}^n \lambda_iE_i + \sum\limits_{i=1}^{n-1}\mu_iD_i & \longmapsto & (\lambda_1, \mu_1, \lambda_2, \ldots, \mu_{n-1},\lambda_n).
\end{array}$$
This allows to consider elements in $\mathrm{D}(2,n,\F)$ as vectors in $\F^{2n-1}$ endowed with the Hamming weight. We are going to look more carefully at their supports in $[2n-1]$. We define the following subsets of $[2n-1]$:
\begin{align*}X_{i,j}&=\{2i-1,2j-1\}, &1\leq i <j \leq n, \\
Y_i&=\{2i\}, & 1\leq i \leq n-1, \\ 
Z_i&=\{2i-1,2i,2i+1\},& 1\leq i \leq n-1.  
\end{align*}
Furthermore, we define $\mathcal N$ to be the collection of such sets, that is 
\begin{equation}\label{eq:Nset} \mathcal N=\{X_{i,j} \,:\, 1\leq i<j \leq n \}\cup \{Y_i \,:\, 1\leq i \leq n-1\} \cup \{Z_i \,:\, 1\leq i \leq n-1\}. \end{equation}
With this notation in mind, we can characterize the elements of $\mathrm{D}(2,n,\F)$ with flag-rank smaller than $3$ in terms of their image under $\varrho$.

\begin{proposition}\label{prop:Hamming_supports}
    Let $A \in \mathrm{D}(2,n,\F)$. Then, $\wf(A)<3$ if and only if there exists some $I\in\mathcal N$ such that 
    $$\supp(\varrho(A))\subseteq I.$$
    In particular, $\C\subseteq \mathrm{D}(2,n,F)$ is an $\ntt{n,2n-4,\geq 3}$ code if and only if there is no nonzero codeword in $\varrho(\C)$ whose support is contained in one of the sets in $\mathcal N$. 
\end{proposition}

\begin{proof} 
 Let $A\in \mathrm{D}(2,n,\F)$, and let us write it as 
 $$A=\sum\limits_{i=1}^n\lambda_iE_i+\sum\limits_{i=1}^{n-1}\mu_iD_i.$$
 Via the isomorphism $\varrho$, we can see that 
 $$J:=\supp(\varrho(A))=\{2i-1 \,:\,\lambda_i \neq 0\}\cup \{2i \,:\, \mu_i\neq 0\}.$$
 It is straightforward to verify that if $J\subseteq I$ for some $I \in \mathcal N$, then $\wf(A)\leq 2$.

 Viceversa, assume that $J$ is not contained in any $I\in \mathcal N$. Denote by $\mathfrak E$ the set of natural even numbers and by $\mathfrak O$ the set of natural odd numbers. Then   one of the following three cases occurs. 

 \noindent {\textbf{Case I:} $|J\cap \mathfrak O|\ge 3$}. This means that there are at least three indices $i_1<i_2<i_3$ such that $\lambda_{i_1},\lambda_{i_2},\lambda_{i_3}\neq 0$. Thus $A_{[i_1]},A_{[i_2]},A_{[i_3]}$ are nonzero matrices, and their rank is at least $1$, implying $\wf(A)\geq 3$.

 \noindent {\textbf{Case II:} $|J\cap \mathfrak E|\ge 2$}. This means that there are at least two indices $i_1<i_2$ such that $\mu_{i_1},\mu_{i_2}\neq 0$. Then, the matrices $A_{[i_1]},A_{[i_2]},A_{[i_2+1]}$ are nonzero and their rank is at least $1$, implying $\wf(A)\geq 3$.

 \noindent {\textbf{Case III:} $J\cap \mathfrak E=\{2i_1\}$ and $J\cap \mathfrak O=\{2i_2-1\}$ for some $i_2\neq i_1,i_1+1$}.
This means that  the matrices $A_{[i_1]},A_{[i_1+1]},A_{[i_2]}$ are nonzero and their rank is at least $1$, implying $\wf(A)\geq 3$.
\end{proof}

Note that for $\delta=3$ the codimension of an $\ntt{n,f(n,3),3}$ MFRD code is
$$g(3)=2+ \lfloor{\sqrt{2}}\rfloor \cdot \left\lfloor{\frac{\sqrt{9}-1}{2}}\right\rfloor=3,$$
and so $f(n,3)=\frac{n(n+1)}{2}-3$.
Thus, by Theorem \ref{thm:diagonal_projection}, constructing an $\ntt{n,f(n,3),3}$ code is equivalent to construct an $\ntt{n,2n-4,3}$ code $\C$ entirely contained $\mathrm{D}(2,n,\F)$. This latter condition can be in turn equivalently reformulated via the condition on $\mathcal D:=\varrho(\C)$ of Proposition~\ref{prop:Hamming_supports}. 
Thus, we are looking for a code satisfying the following property.
\begin{definition} Let $\mathcal N\subseteq 2^{[2n+1]}$ be defined as in \eqref{eq:Nset}.
    A $[2n-1,2n-4]_\F$ code $\mathcal D$ is called 
    \textbf{$\mathcal N$-avoiding} if  for every nonzero $c \in \mathcal D$ there is no $I \in \mathcal N$ with $\supp(c) \subseteq I$. 
\end{definition}

\begin{lemma}\label{lem:avoiding} Let $n\geq 3$.
\begin{enumerate} 
\item 
 The minimum distance of an $\mathcal N$-avoiding $[2n-1,2n-4]_\F$  code  is at least $2$. 
\item The property of a code being $\mathcal N$-avoiding is invariant under multiplication of any coordinate by a nonzero scalar in $\F$.
\end{enumerate}
\end{lemma}

    \begin{proof}
\begin{enumerate}
 \item Every support of size $1$ is contained in some set in $\mathcal N$. Thus, an $\mathcal N$-avoiding code cannot contain codewords of weight $1$.
 \item The property of being $\mathcal N$-avoiding only depends on the supports of the codewords. \qedhere
\end{enumerate}
    \end{proof}

 Let $H \in \Mat(3, 2n-1, \F)$ be a parity-check matrix for an $\mathcal N$-avoiding $[2n-1,2n-4]_\F$ code.  By Lemma \ref{lem:avoiding}(1), $H$ has no zero columns. Moreover, by Lemma \ref{lem:avoiding}(2), we can multiply any column by a nonzero scalar, obtaining again an $\mathcal N$-avoiding code. 

Now, to a parity-check matrix for a $[2n-1,2n-4,\geq 2]_\F$ code, we can associate a vector of $2n-1$ points $(P_1,Q_1,P_2,\ldots,Q_{n-1},P_n)$ in $\PG(2,\F)$, corresponding to the columns of $H$. More precisely, the $(2i-1)$-th column of $H$ will be identified with the point $P_i$, for $i \in [n]$, while the $(2i)$-th column of $H$ will be identified with the point $Q_i$, for $i \in [n-1]$. 

\begin{proposition}\label{prop:Navoiding}
 Let $H\in \Mat(3,2n-1,\F)$ be a parity-check matrix of a $[2n-1,2n-4,\geq 2]_\F$ code $\mathcal D$, and let $(P_1,Q_1,\ldots Q_{n-1},P_n)$ be a vector of points in $\PG(2,\F)$  described as above. Then $\mathcal D$ is $\mathcal N$-avoiding if and only if 
 \begin{enumerate}
     \item $P_i \neq P_j$ for every $i,j \in[n]$ with $i \neq j$,
     \item $Q_i \notin \langle P_i,P_{i+1}\rangle$ for every $i \in [n-1]$.
 \end{enumerate}
\end{proposition}

\begin{proof}
    The supports of the codewords of $\mathcal D$ can be easily seen via the linear dependence of the columns of $H$, or, in other words, via the collinearity relations of the associated projective points $P_1,Q_1,\ldots Q_{n-1},P_n$. From this point of view, we have that:
    \begin{enumerate}
    \item There exists a nonzero codeword of $\mathcal D$ whose support is contained in $X_{i,j}$ if and only if the $(2i-1)$-th and the $(2j-1)$-th columns of $H$ are linearly dependent, that is, if and only if  $P_i=P_j$.
    \item There exists a nonzero codeword of $\mathcal D$ whose support is contained in $Z_{i}$ if and only if the  $(2i-1)$-th, the  $(2i)$-th and  $(2i+1)$-th columns of $H$ are linearly dependent, that is, if and only if $P_i,Q_i$ and $P_{i+1}$ are collinear. Since $P_i \neq P_{i+1}$, tihs is equivalent to say that $Q_i\in \langle P_i,P_j\rangle$.
    \end{enumerate}
    Finally, the supports of codewords in $\mathcal D$ cannot be contained in any $Y_i$, since we are assuming $\mathcal D$ to have minimum distance at least $2$.
\end{proof}

We can now illustrate our construction for MFRD codes with $\delta=3$.

\begin{construction}\label{constr:d=3}
    Let $n\leq |\F|^2+|\F|+1$, and take $n$ distinct points $\{P_1=[u_1],\ldots, P_n=[u_n]\}\subseteq \PG(2,\F)$. Furthermore, for each $i \in [n-1]$, choose $Q_i=[v_i] \notin \langle P_i,P_{i+1}\rangle$. Construct $H\in \Mat(3,2n-1,\F)$ whose $(2i-1)$-th column is $u_i$ and whose $(2i)$-th column is $v_i$: 
    $$H=(\,u_1\, \mid \, v_1 \, \mid \, u_2 \, \mid  \,\cdots\, \mid \,v_{n-1} \, \mid \, u_n ).$$
    Let $\D$ be the $\mathcal N$-avoiding $[2n-1,2n-4,3]_{\F}$ code whose parity-check matrix is $H$. Define the $\left\{n,\frac{n(n+1)}{2}-3,3\right\}_\F$ code
    $$\C:=\varrho^{-1}(\D) \oplus \overline{\U}(n-2,\F),$$
    where $\overline{\U}(n-2,\F)=\phi_{1,3}({\U}(n-2,\F))$.
    \end{construction}

As a direct consequence of Proposition \ref{prop:Navoiding}, combined with
Proposition \ref{prop:Hamming_supports} and Theorem \ref{thm:diagonal_projection}, we get also necessary and sufficient conditions for the existence of MFRD codes with minimum flag-rank distance $\delta=3$.
\begin{theorem}\label{thm:delta=3}
 The code $\C$ given in Construction \ref{constr:d=3} is an $\ntt{n,f(n,3),3}$ MFRD code. Furthermore,
    an $\ntt{n,f(n,3),3}$ MFRD code exists if and only if $n \leq |\PG(2,\F)|=|\F|^2+|\F|+1$.
\end{theorem}

\begin{example}
 Let $n=4$ and $\F=\F_2$. Consider the $[7,4,3]_{\F_2}$ Hamming code whose parity-check matrix is
 $$H=\begin{pmatrix}
     1 & 0 & 1 & 0 & 1 & 0 & 1 \\
     0 & 1 & 1 & 1 & 0 & 0 & 1 \\
     0 & 1 & 0 & 0 & 1 & 1 & 1     
 \end{pmatrix}\in\Mat(3,7,\F_2).$$
Using  Proposition~\ref{prop:Navoiding} it is immediate to verify that $\D$ is $\mathcal{N}$-avoiding. A basis for $\D$ is given by
$$w_1=(1,0,1,1,0,0,0), \; w_2=(1,0,0,0,1,1,0), \; w_3=(0,1,0,1,0,1,0), \; w_4=(1,0,0,1,0,1,1).$$
Thus, the code $\C=\varrho^{-1}(\D)\oplus \overline{\U}(2,\F)$ is generated by the matrices

$$ \begin{pmatrix}
    1 & 0 & 0 & 0 \\
    0 & 1 & 1 & 0 \\
    0 & 0 & 0 & 0 \\
    0 & 0 & 0 & 0 
\end{pmatrix}, \quad \begin{pmatrix}
     1 & 0 & 0 & 0 \\
    0 & 0 & 0 & 0 \\
    0 & 0 & 1 & 1  \\
    0 & 0 & 0 & 0 
\end{pmatrix}, \quad \begin{pmatrix}
    0 & 1 & 0 & 0 \\
    0 & 0 & 1 & 0 \\
    0 & 0 & 0 & 1 \\
    0 & 0 & 0 & 0 
\end{pmatrix}, \quad \begin{pmatrix}
    1 & 0 & 0 & 0 \\
    0 & 0 & 1 & 0 \\
    0 & 0 & 0 & 1 \\
    0 & 0 & 0 & 1    
\end{pmatrix},$$

$$ \begin{pmatrix}
    0 & 0 & 1 & 0 \\
    0 & 0 & 0 & 0 \\
    0 & 0 & 0 & 0 \\
    0 & 0 & 0 & 0 
\end{pmatrix}, \quad \begin{pmatrix}
    0 & 0 & 0 & 0 \\
    0 & 0 & 0 & 1 \\
    0 & 0 & 0 & 0 \\
    0 & 0 & 0 & 0    
\end{pmatrix}, \quad \begin{pmatrix}
    0 & 0 & 0 & 1 \\
    0 & 0 & 0 & 0 \\
    0 & 0 & 0 & 0 \\
    0 & 0 & 0 & 0    
\end{pmatrix}.$$
\end{example}

%%%%%%%%%%%%%%%%%%%%%%%%%%%%%%%%%%%%%%%%%%%%%%%%%%%%%%%%%%%%%%%%%%%%%%%%%%%%%%%%%%%%%%%%%%%%%%%%%%%%
\subsubsection{MFRD codes with $\delta=4$}

We first start by providing a necessary condition on the field size for which there exists an MFRD code with $\delta=4$, by using Proposition \ref{prop:flag_to_Hamming}. After that, we derive a more refined condition in this special case compared to the one given in Corollary \ref{cor:nonexistence}.
Note that in this case we have that the codimension of an $\ntt{n,f(n,4),4}$ MFRD code is
$$g(4)=3+ \lfloor{\sqrt{3}}\rfloor \cdot \left\lfloor{\frac{\sqrt{11}-1}{2}}\right\rfloor=4.$$ 

\begin{theorem}\label{thm:d=4en<=}
Let $\C$ be an $\ntt{n,f(n,4),4}$ code. Then there exists a cap in $\PG(3,\F)$ of cardinality $n$, that is a set of points no three of which are collinear. In particular, 
$$ n\leq \begin{cases} 8 & \mbox{ if } |\F|=2, \\ 
|\F|^2+1 & \mbox{ otherwise.  }  
\end{cases}$$
\end{theorem}

\begin{proof}
 By Proposition \ref{prop:flag_to_Hamming}, the code $\C\cap \mathrm{D}(1,n,\F)$ is an $[n,\ge n-4,\ge 4]_{\F}$ code. In particular, either it is an MDS code of codimension $3$ or $4$, or it is an $[n,n-4,4]_{\F}$ almost MDS code. In both cases, by taking the columns of a parity-check matrix, the corresponding projective points in $\PG(3,\F)$ (or $\PG(2,\F)$) need to be such that no three of them are collinear, i.e. they form a cap. It is well-known that caps in $\PG(2,\F)$ are arcs, and their size can be at most $|\F|+1$ (or $|\F|+2$ if $|\F|=2^h$). If we are instead in $\PG(3,\F)$,  the largest size that a  cap can have was proved in \cite{bose1947mathematical,quist1952some} to be 
 $$\begin{cases} 8 & \mbox{ if } |\F|=2, \\ 
|\F|^2+1 & \mbox{ otherwise.  }  
\end{cases}$$
\end{proof}

We now present a construction of MFRD code with $\delta=4$.

\begin{construction}\label{constr:C}
Let $n\geq 4$.
Let $\mathcal{C}_1,\ldots,\mathcal{C}_{n-2}$ be $[2\times 2,2,2]_{\mathbb{F}}$ MRD codes  and  let $\mathcal{C}_{n-1}$ be an $[n,n-3,4]_\F$ MDS code.
For any $i \in [n-2]$, denote by $\overline{\mathcal{C}}_i=\phi_{i,i+1}(\mathcal{C}_i)$ and denote by $\overline{\mathcal{C}}_{n-1}$ the 
the image of $\C_{n-1}$ under the natural mapping $\psi:\F^n \rightarrow \mathrm{D}(1,n,\F)$. 
Moreover, let $\overline{\U}(n-3,\F)=\phi_{1,4}(\U(n-3,\F))$.
Define the $\ntt{n,t}$ code
\[ \mathcal{C}=\overline{\mathcal{C}}_1+\ldots+\overline{\mathcal{C}}_{n-2}+\overline{\mathcal{C}}_{n-1}+\overline{\U}(n-3,\F).  \]
\end{construction}

In Figure \ref{pic} we can see the structure of a codeword of Construction \ref{constr:C} when $n=5$, in which in the red, green and blue squares there are the codewords of $\mathcal{C}_1$, $\mathcal{C}_2$ and $\mathcal{C}_3$, respectively, and on the main diagonal there are the codewords of $\mathcal{C}_4$.

\begin{figure}[ht!]
\centering
\begin{tikzpicture}[thick, scale=0.6]
\draw[help lines, thick] (0,0) -- (1,0)-- (1,-1) -- (2,-1) -- (2,-2)-- (3,-2) -- (3,-3)-- (4,-3) -- (4,-4)--(5,-4)--(5,1) -- (0,1)--(0,0);

\draw[gray, thick, fill=yellow!10] (0,1) -| (0,0) -| (1,0) -| (1,1) -| cycle;
%\draw[fill=blue!10] (0,1) -| (0,0) -| (1,0) -| (1,1) -| cycle;
\draw[gray, thick, fill=yellow!10] (1,0) -| (1,-1) -| (2,-1) -| (2,0) -| cycle;
\draw[gray, thick, fill=yellow!10] (2,-1) -| (2,-2) -| (3,-2) -| (3,-1) -| cycle;
\draw[gray, thick, fill=yellow!10] (3,-2) -| (3,-3) -| (4,-3) -| (4,-2) -| cycle;
\draw[gray, thick, fill=yellow!10] (4,-3) -| (4,-4) -| (5,-4) -| (5,-3) -| cycle;

\draw[red, thick] (1,1)-|(1,-1)-|(3,-1)-| (3,1) -|cycle;
\draw[ForestGreen, thick] (2,0)-|(2,-2)-|(4,-2)-| (4,0) -|cycle;
\draw[blue!70, thick] (3,-1)-|(3,-3)-|(5,-3)-| (5,-1) -|cycle;
\end{tikzpicture}
\caption{Codewords from Construction \ref{constr:C} for $n=5$. In the boxes with the colored background we embed the codewords from the MDS code $\C_{4}$. In the boxes with the colored edges we sum codewords from $\C_1,\C_2,\C_3$, respectively.}
\label{pic}
\end{figure}
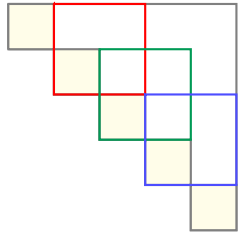

\begin{theorem}\label{th:constrd4}
Let $n\in\mathbb N$ with $n\geq 4$. Let $\mathcal{C}_1,\ldots,\mathcal{C}_{n-2}$ be $[2\times 2,2,2]_{\mathbb{F}}$ MRD codes  and  let $\mathcal{C}_{n-1}$ be an $[n,n-3,4]_\F$ MDS code. The code  
\[\mathcal{C}=\overline{\mathcal{C}}_1+\ldots+\overline{\mathcal{C}}_{n-2}+\overline{\mathcal{C}}_{n-1}+\overline{\U}(n-3,\F) \]
defined in Construction \ref{constr:C} is an $\ntt{n,f(n,4),4}$ MFRD code.
\end{theorem}
\begin{proof}

First of all we prove that $\dim_{\F}(\C)=\frac{n(n+1)}{2}-4$. Since $\overline{\mathcal{C}}_1+\ldots+\overline{\mathcal{C}}_{n-2}+\overline{\mathcal{C}}_{n-1}\subseteq \mathrm{D}(3,n,\F)$ then $(\overline{\mathcal{C}}_1+\ldots+\overline{\mathcal{C}}_{n-2}+\overline{\mathcal{C}}_{n-1})\cap \overline{\U}(n-3,\F)=\{0\}$. Therefore, we only need to show that $\dim_{\F}(\overline{\mathcal{C}}_1+\ldots+\overline{\mathcal{C}}_{n-2}+\overline{\mathcal{C}}_{n-1})=2(n-2)+n-3$. 
Since $\dim_\F(\overline{\C}_1)+\ldots+\dim_\F(\overline{\C}_{n-1})=\dim_\F(\C_1)+\ldots+\dim_\F(\C_{n-1})=2(n-2)+n-3$, it is enough to show that $\overline{A}^{(1)},\ldots,\overline{A}^{(n-1)}$ are $\F$-linearly independent, for every choice of $\overline{A}^{(i)} \in {\overline{\mathcal{C}}}_i\setminus\{0\}$ for  $i\in[n-1]$. 
Let $A^{(i)} \in {\mathcal{C}}_i$ be the element corresponding to $\overline{A}^{(i)} \in {\overline{\mathcal{C}}}_i$, for every $i$.
Let $\lambda_1,\ldots,\lambda_{n-1} \in \F$ be such that
\begin{equation}\label{eq:B=0} 
B=\lambda_1 \overline{A}^{(1)}+\ldots+\lambda_{n-1}\overline{A}^{(n-1)}=0. 
\end{equation}
Let $\mathcal H:=\{i \in [n-2] \st \lambda_i \neq 0\}$. Towards a contradiction, assume that $\mathcal H\neq \emptyset$ and let $\kappa:=\min(\mathcal H)$. Thus, by \eqref{eq:B=0} we have that 
$$(0,0)=(B_{\kappa,\kappa+1},B_{\kappa,\kappa+2})=\lambda_\kappa (A^{(\kappa)}_{1,1},A^{(\kappa)}_{1,2}). $$
However, $A^{(\kappa)}\in \C_\kappa\setminus\{0\}$ is a $2\times 2$ matrix of rank $2$ and its first row cannot be $0$. Thus, $\lambda_{\kappa}=0$, yielding a contradiction. 
This means $\mathcal H=\emptyset$ and $B=\lambda_{n-1}\overline{A}^{(n-1)}$. This means that
$$\lambda_1=\ldots=\lambda_{n-1}=0,$$
showing that $\overline{A}^{(1)},\ldots,\overline{A}^{(n-1)}$ are $\F$-linearly independent and 
\[ \dim_{\F}(\mathcal{C})=2(n-2)+n-3+\dim_{\F}(\U(n-3,\F))=\frac{n(n+1)}2-4. \]

\noindent We now show that $\df(\C)\geq4$.  Any element  $C\in\mathcal{C}\setminus \{0\}$ can be uniquely written as 
\[ C=\overline{A}^{(1)}+\ldots+\overline{A}^{(n-1)}+\overline{A}^{(n)}, \]
where ${A}^{(i)} \in {\mathcal{C}}_i$ for any $i \in [n-1]$ and $\overline{A}^{(n)} \in \overline{U}(n-3,\F)$.
Our aim is to first prove that $\w_{\ff}(C)\geq 4$. By Theorem \ref{thm:diagonal_projection} we may assume that $\overline{A}^{(n)}=0$, which implies that the nonzero entries of $C$ are contained in the first three main diagonals. 
Let $\mathcal J:=\{ \ell  \in [n-2] \st \overline{A}^{(\ell)} \neq 0 \}$.
We divide the discussion in three cases.\\
\textbf{Case 1:} $|\mathcal{J}|\geq 2$. \\
Let $i=\min(\mathcal{J})$ and $j=\max(\mathcal{J})$.
In this case, the $i$-th row and the $(j+2)$-th column of $C$ are nonzero, therefore $\mathrm{rk}(C_{[i]})$, $\mathrm{rk}(C_{[i+1]})$, $\mathrm{rk}(C_{[j+1]})$ and $\mathrm{rk}(C_{[j+2]})$ are at least one and so $\w_{\ff}(C)\geq \mathrm{rk}(C_{[i]})+ \mathrm{rk}(C_{[i+1]})+\mathrm{rk}(C_{[j+1]})+\mathrm{rk}(C_{[j+2]})\geq 4$.\\
\textbf{Case 2:}  $\mathcal{J}=\{i\}$.\\
If $\overline{A}^{(n-1)}=0$ then $\w_{\ff}(C)=\w_{\ff}(\overline{A}^{(i)})=4$. Otherwise, on the main diagonal of $C$ there are at least three nonzero elements in the entries $(i_1,i_1)$, $(i_2,i_2)$ and $(i_3,i_3)$ for $i_1,i_2,i_3 \neq i+1$. Moreover, the first nonzero row of $\overline{A}^{(i)}$ implies that $\mathrm{rk}(C_{[i+1]})$ is at least one. 
Therefore, $\w_{\ff}(C)\geq \mathrm{rk}(C_{[i_1]})+ \mathrm{rk}(C_{[i_2]})+\mathrm{rk}(C_{[i_3]})+\mathrm{rk}(C_{[i+1]})\geq 4$.\\ 
\textbf{Case 3:}  $\mathcal{J}=\emptyset$.\\
In this case $C=\overline{A}^{(n-1)}$ and hence, since the minimum Hamming distance of $\mathcal{C}_{n-1}$ is $4$, then there are at least four nonzero entries on the main diagonal of $C$, implying $\w_{\ff}(C)\geq 4$.\\
Finally, observe that $\dim_{\F}(\mathcal{C})=\frac{n(n+1)}2-4$ reaches the equality in the Singleton-like bound of Proposition \ref{pr:Singleton} when $\delta=4$ and therefore the minimum flag-rank distance is $4$.
\end{proof}

\begin{remark}
Construction \ref{constr:C} works also for $n=3$, by taking as MDS code $\C_{n-1}=\{0\}$. In this case, it coincides with Construction \ref{constr:A}.
\end{remark}

In Example \ref{ex:constC} we illustrate how to use Construction \ref{constr:C} to obtain an MFRD code.

\begin{example}\label{ex:constC}
    Let $n=4$. We consider the $[2\times 2,2,2]_{\F_2}$ MRD codes $\C_1=\C_2$ generated by
    \begin{align*}
        A_1=\begin{pmatrix}
            1 & 0 \\
            0 & 1
        \end{pmatrix}, \;\; A_2=\begin{pmatrix}
            0 & 1 \\
            1 & 1
        \end{pmatrix}. 
    \end{align*}
    Let $C_3$ be the $[4,1,4]_{\F_2}$ MDS code generated by $G=\begin{pmatrix}
        1 & 1 & 1 & 1
    \end{pmatrix}$. 
    Finally we have that $\overline{\U}(1,\F_2)$ is generated by
    $$A:=\begin{pmatrix}
        0 & 0 & 0 & 1 \\
        0 & 0 & 0 & 0 \\
        0 & 0 & 0 & 0\\
        0 & 0 & 0 & 0
    \end{pmatrix}.$$
    Consider then the codes $\overline{\C}_1=\phi_{1,2}(\C_1)$ and  $\overline{\C}_2=\phi_{2,3}(\C_2)$ generated respectively by 
   \begin{align*}
     \phi_{1,2}(A_1)=\begin{pmatrix}
         0 & 1 & 0 & 0 \\
         0 & 0 & 1 & 0 \\
         0 & 0 & 0 & 0 \\
         0 & 0 & 0 & 0
     \end{pmatrix}, \; \phi_{1,2}(A_2)=\begin{pmatrix}
         0 & 0 & 1 & 0 \\
         0 & 1 & 1 & 0 \\
         0 & 0 & 0 & 0 \\
         0 & 0 & 0 & 0
     \end{pmatrix}, \; 
    \end{align*}
    and 
    \begin{align*}
     \phi_{2,3}(A_1)=\begin{pmatrix}
         0 & 0 & 0 & 0 \\
         0 & 0 & 1 & 0 \\
         0 & 0 & 0 & 1 \\
         0 & 0 & 0 & 0
     \end{pmatrix}, \; \phi_{2,3}(A_2)=\begin{pmatrix}
         0 & 0 & 0 & 0 \\
         0 & 0 & 0 & 1 \\
         0 & 0 & 1 & 1 \\
         0 & 0 & 0 & 0
     \end{pmatrix}. \; 
    \end{align*}
    We let $\overline{\mathcal{C}}_3$ be the $\F_2$-subspace of diagonal matrices whose diagonal are the codewords of $C_3$, i.e. the code generated by
    $$\overline{G}=\begin{pmatrix}
        1 & 0 & 0 & 0\\
        0 & 1 & 0 & 0\\
        0 & 0 & 1 & 0 \\
        0 & 0 & 0 & 1
    \end{pmatrix}.$$
    By Theorem \ref{th:constrd4}, the space 
    $$\langle A,  \phi_{1,2}(A_1),  \phi_{1,2}(A_2), \phi_{2,3}(A_1), \phi_{2,3}(A_2),\overline{G}\rangle_{\F_2}\subseteq \U(4,\F_2)$$
    is a $\{4,6,4\}_{\F_2}$ MFRD code.
\end{example}

\begin{remark}\label{rem:fieldsize} 
The construction provided in Theorem \ref{th:constrd4} works whenever the following two conditions hold.
\begin{enumerate}
    \item There exists an $[n,n-3,4]_\F$ MDS code.
    \item There exists a $[2\times 2,2,2]_{\F}$ MRD code.
\end{enumerate}
For (1), we see that by duality this holds if and only if there exists an $[n,3,n-2]_\F$ MDS code. This is a special case of the MDS conjecture, which in dimension $3$  was shown to be true. In particular, such a code exists if and only if 
$$|\F| \geq \begin{cases} n-2 & \mbox{ if } n= 2^h+2, \\
n-1 & \mbox{ otherwise. }
\end{cases} $$

On the other hand, a $[2\times 2,2,2]_\F$ MRD code exists if and only if there exists a degree $2$ field extension of $\F$. Indeed, by assuming, without loss of generality, that the $2$-dimensional rank-metric code is generated by $\mathrm{Id}_2$ and $A$, then it is MRD if and only if $A$ has no eigenvalues in $\F$, that is, if and only if its characteristic polynomial is irreducible over $\F$.
\end{remark}

\subsection{Duality of maximum flag-rank distance codes}
  For classical metric spaces used in coding theory, codes meeting the Singleton bounds with equality have the remarkable property of being closed under duality. More precisely, for codes endowed with the Hamming metric and codes endowed with the rank metric, one can consider the duality  given by the standard inner product (or by any isometric equivalent bilinear form), and study the dual spaces with respect to such a nondegenerate bilinear form. In this situation, the dual of a code meeting the Singleton bound with equality -- that is, an MDS code for the Hamming metric, and an MRD code for the rank metric -- meets again the Singleton bound with equality.

 However, surprisingly this is not the case in general for the flag-rank distance. Let $\varphi:\U(n,\F)\times \U(n,\F) \rightarrow \F$ be any nondegenerate $\F$-bilinear form, and consider
  the \textbf{dual code} 
 $$\C^{\perp_\varphi}:=\left\{ N \in \U(n,\F) \st \varphi (M,N)=0 \mbox{ for every } M \in \C\right\}.$$
 
 First of all, if we take any generic nondegenerate bilinear form $\varphi$, we can observe that the dual of an MFRD code has not always codimension of the form $g(\delta)$, for some $\delta\in \mathbb N$. Hence, the dual of an MFRD code cannot be MFRD according to our Definition \ref{def:MFRD}. Thus, in order to have some duality statement, we would need to allow a weaker definition of MFRD code. The most natural one is the following. For a given $t$, let 
 $$\delta(t):=\min\{\delta \st f(n,\delta)\ge  t \},$$
 and define a \textbf{quasi MFRD code} to be an $\ntt{n,t,\delta(t)}$ code. In other words, we can define a quasi MFRD code as the code that for the given dimension has the maximum possible minimum distance according to the Singleton bound of Proposition \ref{pr:Singleton}. With this in mind, we can ask ourselves if the dual of an MFRD codes is quasi MFRD.

However, the answer is again negative if we restrict to study the case of $\varphi$ being the standard inner product
$$\varphi(M,N):=\sum_{1\le i,j\le n}M_{i,j}N_{i,j},\qquad \mbox{ for any } M,N \in \U(n,\F).$$
In this case the dual code of an MFRD code is not necessarily MFRD, nor quasi MFRD. Indeed,  let $\C$ be an an optimum distance MFRD code from Construction \ref{constr:A}. Then, its dual code $\C^{\perp_\varphi}$ contains elements of flag-rank weight $1$. Thus, $\C^{\perp_\varphi}$ has codimension $n-k+1$ and minimum flag-rank distance $1$. Clearly, for minimum flag-rank distance $1$, the only MFRD code is the whole space $\U(n,\F)$. Furthermore,  for the given codimension $n-k+1$ we can construct flag-rank-metric codes with higher minimum flag-rank distance. Indeed, any codimension $n-k$ subcode of the code given in Section \ref{sec:d=2} has minimum flag-rank distance at least $2$. This shows that $\C^{\perp_\varphi}$ is not even quasi MFRD. A similar situation happens for any MFRD code obtained from Construction \ref{constr:B}. 
 The reader can also verify that the dual code $\C^{\perp_\varphi}$ when $\C$ is the $\left\{n,\frac{n(n+1)}{2}-1,2\right\}$ code given in Section \ref{sec:d=2} is a $\ntt{n,1,2}$ code, and hence very far from being MFRD.
 
 In conclusion, if some weaker form of a duality statement for MFRD codes may hold, then the defining bilinear form must be different from the standard inner product. In order to understand this, an investigation on the $\F$-linear isometries of $(\U(n,\F),\df)$ would be certainly beneficial. We leave this to future research.

\section{Conclusions}\label{sec:conclusion}
In this paper we pursued the study of linear flag-rank-metric codes initiated by Fourier and Nebe in \cite{fourier2021degenerate}. We started by extending the Singleton-like bound to any possible value of the minimum flag-rank distance, originally provided in \cite{fourier2021degenerate} only for the optimum-distance case.
We focused then on constructions of maximum flag-rank distance codes.  Table \ref{tab:existence} summarizes the constructions and existence results   obtained in Section \ref{sec:MFRC}.

\begin{center}
\begin{table}[htp]
\tabcolsep=1 mm
\renewcommand{\arraystretch}{2}
\begin{tabular}{|c|c|c|c|c|}
\hline
\hspace{0.2cm}$n$\hspace{0.2cm} & $\delta$ & \mbox{Conditions} & \mbox{References} \\ \hline
$2k-1+\epsilon$ &  $k(k+\epsilon)$ & $\begin{array}{ccc} \epsilon= \begin{cases} 1 & \text{if}\,\,n=2k\\ 0 & \text{if}\,\,n=2k-1 \end{cases}\\ 
\mbox{Existence of a } [(k+\epsilon)\times(k+\epsilon),k+\epsilon,k+\epsilon]_{\mathbb{F}} \mbox{ MRD code}
\end{array}$ & Th. \ref{th:extopt} \\ \hline
$2k-1$ &  $k^2-1$ & Existence of a $[(k+1)\times(k+1),k+1,k+1]_{\mathbb{F}}$ MRD code & Th. \ref{th:existnoddmrd} \\ \hline
 &  2 &  & Sec. \ref{sec:d=2}  \\ \hline
 &  3 & $n\leq |\mathbb{F}|^2+|\mathbb{F}|+1$ & Th. \ref{thm:delta=3}  \\ \hline
 &  4 & $\begin{array}{ccc} |\F| \geq \begin{cases} n-2 & \mbox{ if } n= 2^h+2 \\
n-1 & \mbox{ otherwise }
\end{cases}\\ \mbox{Existence of a } [2\times 2,2,2]_{\mathbb{F}} \mbox{ MRD code} \end{array}$ & Th. \ref{th:constrd4}  \\ \hline
\end{tabular}
\caption{Existence of a maximum $\{n,t,\delta\}_{\mathbb{F}}$ flag-rank distance codes.}
\label{tab:existence}
\end{table}
\end{center}

Our computational results indicate the existence of MFRD codes also for other values of $\delta$, which are not covered by our constructions. Example~\ref{ex:delta5} illustrates an MFRD code with~$\delta=5$.

\begin{example}\label{ex:delta5} 
Consider $n=5$ and $\delta=5$. The $\F_q$-subspace generated by the following nine matrices
\[ \small{\left( \begin{matrix} 
0 & 0 & 1 & 1 & 1\\
0 & 1 & 0 & 0 & 0\\
0 & 0 & 0 & 0 & 0\\
0 & 0 & 0 & 1 & 1\\
0 & 0 & 0 & 0 & 1 \end{matrix} \right),
\left(
\begin{matrix}
1 & 0 & 1 & 0 & 0\\
0 & 0 & 1 & 1 & 1\\
0 & 0 & 0 & 1 & 0\\
0 & 0 & 0 & 0 & 1\\
0 & 0 & 0 & 0 & 1
\end{matrix}
\right),
\left(
\begin{matrix}
0 & 1 & 0 & 0 & 0\\
0 & 1 & 1 & 1 & 0\\
0 & 0 & 1 & 1 & 1\\
0 & 0 & 0 & 1 & 1\\
0 & 0 & 0 & 0 & 1
\end{matrix}
\right),
\left(
\begin{matrix}
1 & 1 & 0 & 0 & 0\\
0 & 1 & 1 & 0 & 0\\
0 & 0 & 1 & 1 & 0\\
0 & 0 & 0 & 1 & 0\\
0 & 0 & 0 & 0 & 0
\end{matrix}
\right),}
\]
\[\small{
\left(
\begin{matrix}
1 & 0 & 1 & 0 & 1\\
0 & 0 & 0 & 1 & 0\\
0 & 0 & 0 & 0 & 0\\
0 & 0 & 0 & 1 & 0\\
0 & 0 & 0 & 0 & 1
\end{matrix}
\right),
\left(
\begin{matrix}
1 & 1 & 1 & 0 & 1\\
0 & 0 & 1 & 0 & 0\\
0 & 0 & 0 & 0 & 1\\
0 & 0 & 0 & 0 & 0\\
0 & 0 & 0 & 0 & 0
\end{matrix}
\right),
\left(
\begin{matrix}
1 & 0 & 1 & 1 & 1\\
0 & 0 & 1 & 0 & 1\\
0 & 0 & 0 & 0 & 0\\
0 & 0 & 0 & 0 & 1\\
0 & 0 & 0 & 0 & 1
\end{matrix}
\right),
\left(
\begin{matrix}
1 & 0 & 0 & 0 & 0\\
0 & 0 & 1 & 0 & 0\\
0 & 0 & 0 & 1 & 0\\
0 & 0 & 0 & 0 & 0\\
0 & 0 & 0 & 0 & 1
\end{matrix}
\right),
\begin{pmatrix}
1 & 1 & 0 & 0 & 0\\
0 & 1 & 1 & 1 & 0\\
0 & 0 & 1 & 1 & 0\\
0 & 0 & 0 & 0 & 1\\
0 & 0 & 0 & 0 & 1
\end{pmatrix}}
 \]
is a  $\{5,9,5\}_{\mathbb{F}_q}$ maximum flag-rank distance code when $q \in \{2,4\}$.
\end{example}

Note that the nine matrices in Example \ref{ex:delta5}  {do not generate} an MFRD code when $q \in \{3,5\}$. However, with the aid of \textsc{magma}, we found sporadic examples of $\{5,9,5\}_{\mathbb{F}_q}$ MFRD code also when $q \in \{3,5\}$.

\medskip

Our study gives rise to plenty of further research lines. We list a few of them.
\begin{enumerate}
    \item[(1)] The Singleton-like bound of Theorem \ref{th:Singleton_codimension} is not tight in general. Providing other constructions for MFRD codes with some values of $\delta$ will establish other instances where the mentioned bound is met with equality. 
    \item[(2)] On the other hand, for $\delta=4$ the Singleton-like bound is not met with equality when the field size is small compared to $n$, as shown in Theorem \ref{thm:d=4en<=}. This motivates the investigation of new bounds which depend on the field size, in the spirit of what happens with the Griesmer bound for codes in the Hamming metric. 
    \item[(3)] Again on the Singleton-like bound, Remark \ref{rem:nonexistence} highlights the fact there are no optimum-distance MFRD codes in $\U(n,\F)$ over algebraically closed fields, for $n\geq 3$. Thus, over an algebraically closed field $\overline{\F}$ there must be a stronger bound. A natural bound on the parameters of an $\{n,t,\delta\}_{\overline{\F}}$ is 
    $$ t\leq \frac{n(n+1)}{2}- \dim(\mathcal V_{\delta-1}^n),$$
    where $\mathcal V_{\delta-1}^n$ is the variety of all the elements in $\U(n,\overline{\F})$ of flag-rank weight at most $\delta-1$. 
    This is the same idea that yields a stronger bound for rank-metric codes over algebraically closed fields, which was also shown to be tight for every set of parameters; see \cite{roth1991maximum}. Thus, it would be interesting to determine the dimension of the algebraic variety $\mathcal V_{\delta-1}^n$ and study whether the new derived bound is tight. A first easy case we can derive is when $n=2k-1$ and $\delta=k^2$ (that is, optimum-distance flag-rank-metric codes for odd $n$). By Remark \ref{rem:nonexistence}, the variety $\mathcal V_{k^2}^{2k-1}$ is the direct sum of two copies of $\U(k-1,\overline{\F})$ (the top-left block and the bottom-right one) with the determinantal variety $\mathcal D_{k-1}^{k}$ of all $k\times k$ matrices over $\overline{\F}$ with rank at most $k-1$. The latter is known to have dimension $k^2-1$ (\cite[Theorem 2.1]{westwick1972spaces}), and, therefore, for a $\{2k-1,t,k^2\}_{\overline{\F}}$ code, we deduce that 
    $$ t \leq k(2k-1)-(k-1)k-(k^2-1)=1.$$
    \item[(4)] Analyze and possibly characterize the $\F$-linear isometries of $(\U(n,\F), \dd_{\textnormal{fr}})$.
    \item[(5)] Remark \ref{rem:fieldsize} leaves  different questions {open}.
\begin{itemize}
\item Is it possible to construct MFRD codes with $\delta=4$ over a field $\F$ where $\sqrt{n-1}\le |\F|<n-1$ (or $|\F|<n-2$ when $n=2^h+2$)?
\item Is it possible to construct MFRD codes with $\delta=4$ over a field $\F$ that has no degree $2$ extension? This is for instance the case of algebraically closed field, as well as other pathological cases.
\end{itemize}
\end{enumerate}

\bigskip

\bibliographystyle{abbrv}
\bibliography{references.bib}

\end{document}